\documentclass[a4paper]{amsart}

\usepackage{tikz}
\usepackage{tkz-graph}
\usetikzlibrary{arrows,positioning,automata}
\usepackage{algorithm}
\usepackage{algorithmicx}
\usepackage[noend]{algpseudocode}
 
\makeatletter
\def\BState{\State\hskip-\ALG@thistlm}
\makeatother

\newtheorem{theorem}{\bf Theorem}[section]

\newtheorem{corollary}{\bf Corollary}[section]
\newtheorem{definition}{\bf Definition}[section]
\newtheorem*{thmain}{\bf Main Theorem}
\newtheorem{remark}{\bf Remark}[section]
\newtheorem{example}{\bf Example}[section]
\newtheorem{simpexamples}{\bf Simplest examples}[section]
\newtheorem{proposition}{\bf Proposition}[section]

\newcommand{\dist}{\operatorname{dist}}
\newcommand{\erre}{\mathbb{R}}
\newcommand{\hess}{\operatorname{Hess}}
\newcommand{\evA}{\operatorname{evA}}
\newcommand{\eva}{\operatorname{eva}}
\newcommand{\area}{\operatorname{A}}

\begin{document}

\title{A lower Bound for the Area of Plateau Foams}
\author{
V. Gimeno$^{1}$, S. Markvorsen$^{2}$, and\\ J. M. Sotoca$^{3}$}

\address{$^{1}$Departament de Matem\`{a}tiques-IMAC, Universitat Jaume I, Castell\'o,
Spain. \\
$^{2}$DTU Compute, Technical University of Denmark.\\
$^{3}$Departamento de Lenguajes y Sistemas Inform\'aticos-INIT, Universitat Jaume I, Castell\'o,
Spain.}



\keywords{foams, bubbles, density, pressure, comparison geometry}

\begin{abstract}
Real foams can be viewed as a geometrically well-organized dispersion of more or less spherical bubbles in a liquid. When the foam is so drained that the liquid content  significantly decreases, the bubbles become polyhedral-like and the foam can be viewed now as a network of thin liquid films intersecting each other at the Plateau borders according to the celebrated Plateau's laws.

In this paper we estimate from below the surface area of a spherically bounded piece of a foam. Our  main tool is a new version of the divergence theorem which is adapted to the specific geometry of a foam with special attention to its classical Plateau singularities.

As a benchmark application of our results we obtain lower bounds for the fundamental cell of a Kelvin foam, lower bounds for the so-called cost function, and for the difference of the pressures appearing in minimal periodic foams. Moreover, we provide an algorithm whose input is a set of isolated points in space and whose output is the best lower bound estimate for the area of a foam that contains the given set as its vertex set.
\end{abstract}











\maketitle

\section{Introduction} \label{secIntro}
A foam is a cell decomposition of the Euclidean $3$-space $\erre^3$ into a finite or infinite number of properly embedded, connected 3D chambers. Here, the chambers are not assumed to be either compact or homeomorphic to a ball in $\mathbb{R}^{3}$, but to qualify as a foam we will assume that they comply with the famous Plateau rules.

The Plateau rules are the following: Firstly, the interfaces between neighbouring chambers all have constant mean curvature; secondly the interfaces meet in threes (at equal $2\pi/3$ angles) along smooth edges; and thirdly the edges always meet four at a time in isolated points, where the angle between any pair of edges is precisely $\arccos(-1/3)$.

A  foam is clearly represented by its $1$-, $2$-, and $3$-skeleton, i.e., the union of its faces, edges, and vertices. We will typically denote this union by $F$.
The constant mean curvature of the surface of each face in a foam is proportional to the pressure difference between the two cells meeting along the face. Every foam is organized around two angles: Every vertex treats the foam like the center vertex that locates the $6$ inner wings in a regular tetrahedron, and every edge organizes three of these wings to have equal angles $2\pi/3$ between them, see Figure \ref{tetraball}.

Plateau found these rules in the nineteenth century when he was studying the geometry of assembled bubbles in equilibrium. From then on, foams were largely studied because of their amazing physical and mathematical properties related e.g. to packings and isoperimetric problems, see \cite{Taylor, H, M}.

The goal of this paper is to obtain lower bounds for the area of 'spherical scoops' of a foam. Locally the area of the tetrahedral linear foam (totally geodesic) inside a sufficiently small ball of radius $R$ centered at a foam vertex is similar to  $\pi \theta_{v}\,R^{2}$, where $\theta_{v} = \frac{3}{\pi}\arccos(-1/3)$ is the density of the tetrahedral $6$-wing construction at the vertex, see figure \ref{tetraball}. This is the estimated area that we will apply for comparison at every vertex of the given foam in order to achieve an effective lower bound for the total area of a foam.

\begin{figure}
\centerline{
\includegraphics[height=30mm]{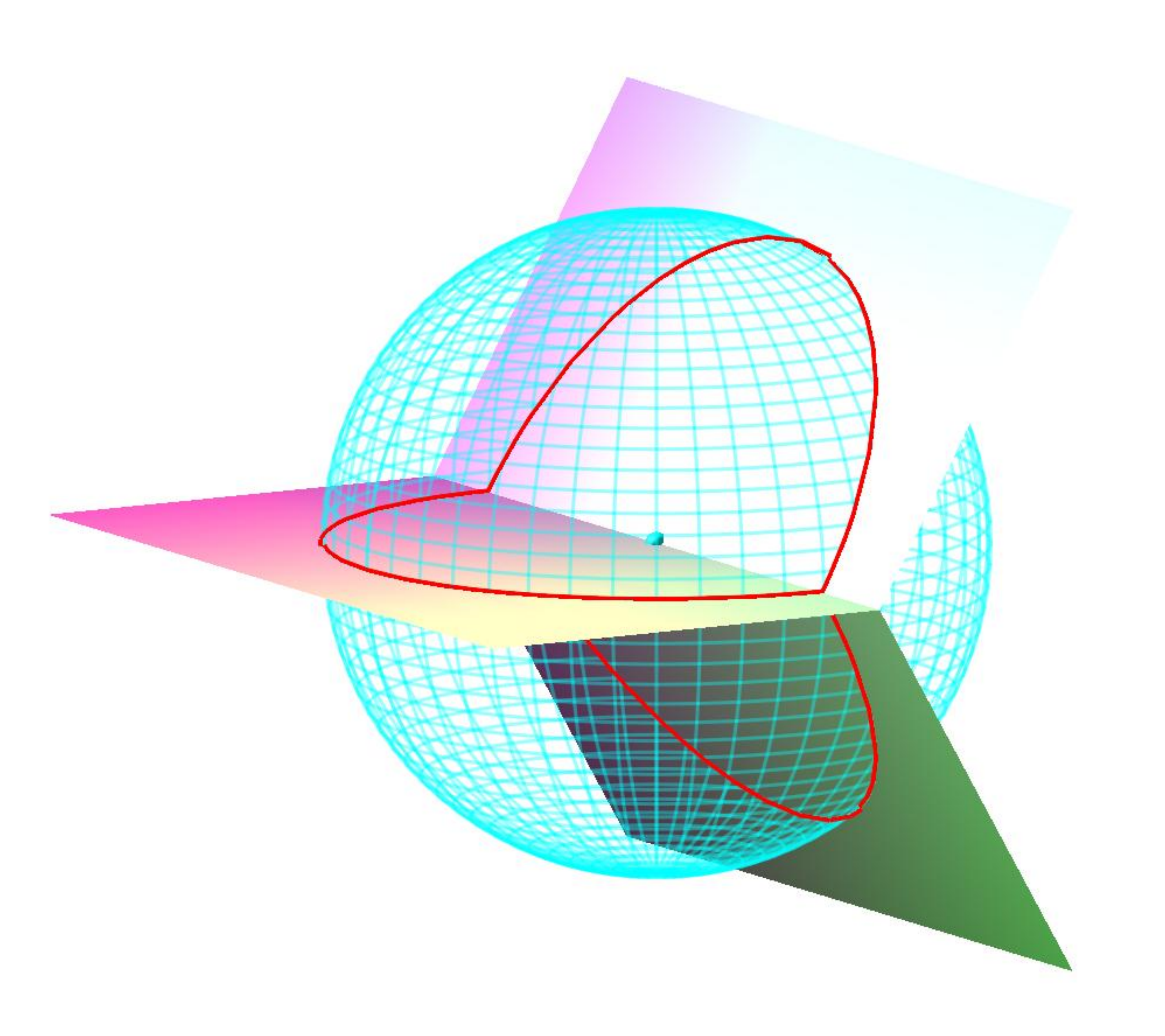} \quad
\includegraphics[height=30mm]{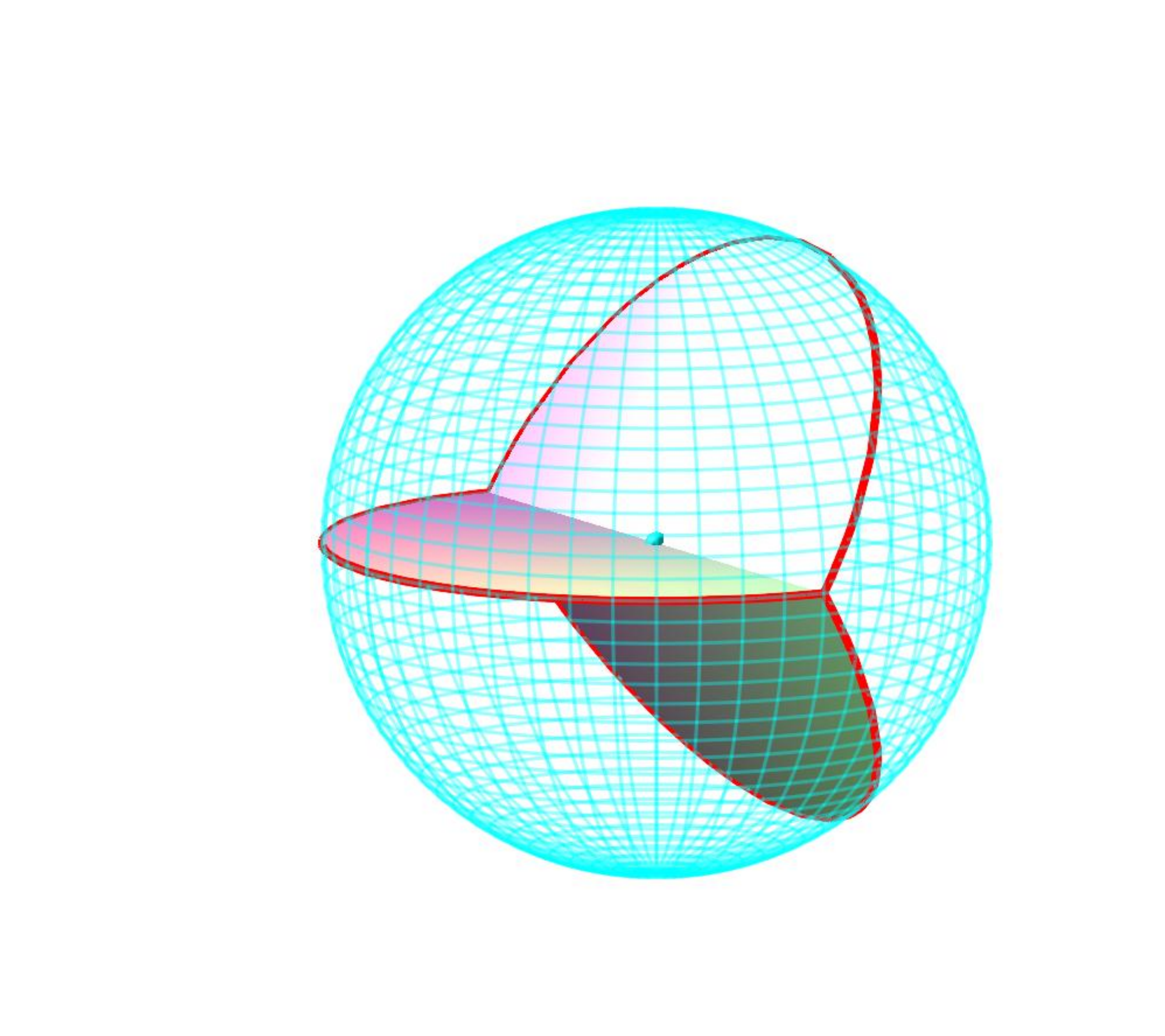} \quad
\includegraphics[height=30mm]{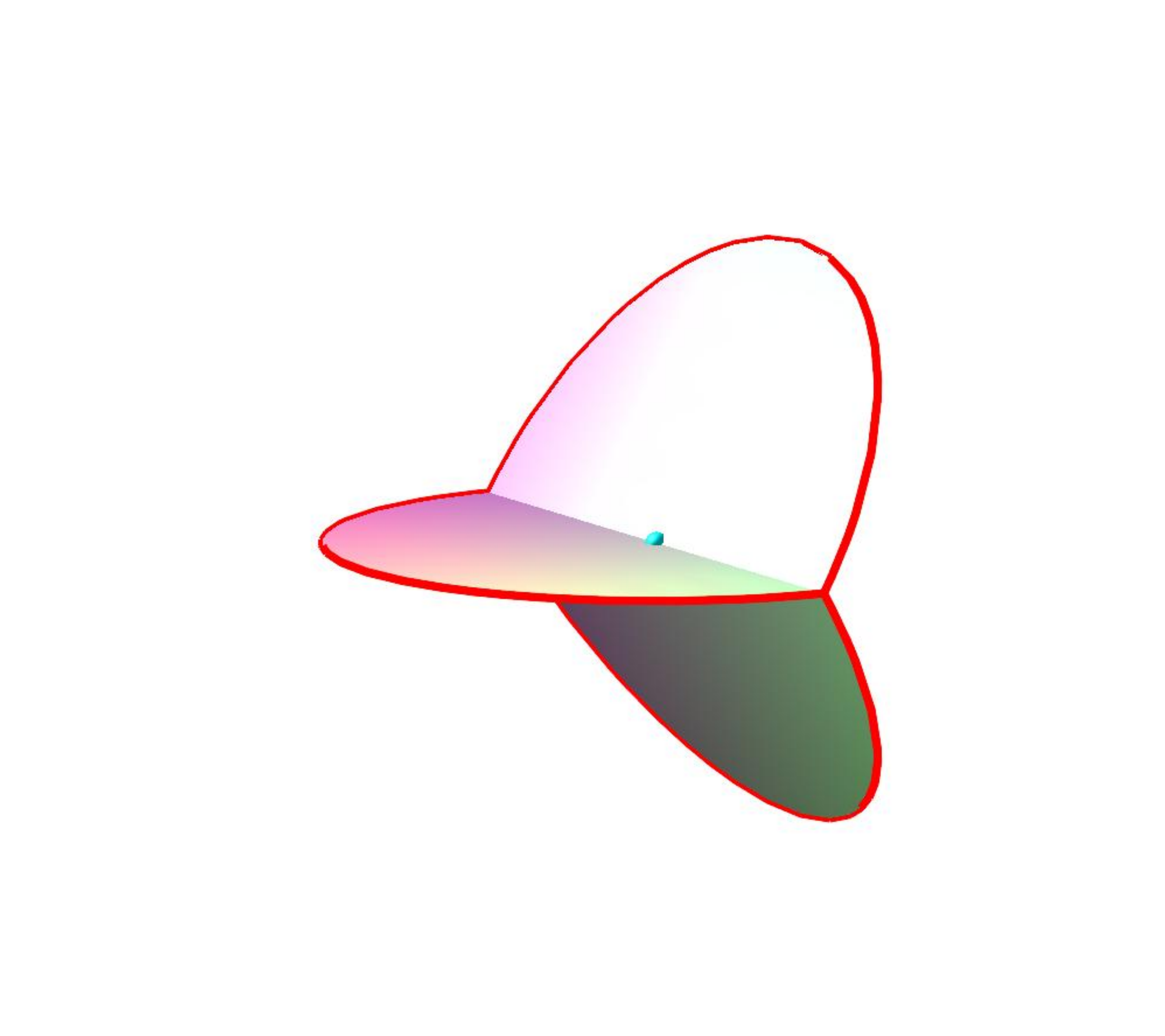}}
\centerline{
\includegraphics[height=32mm]{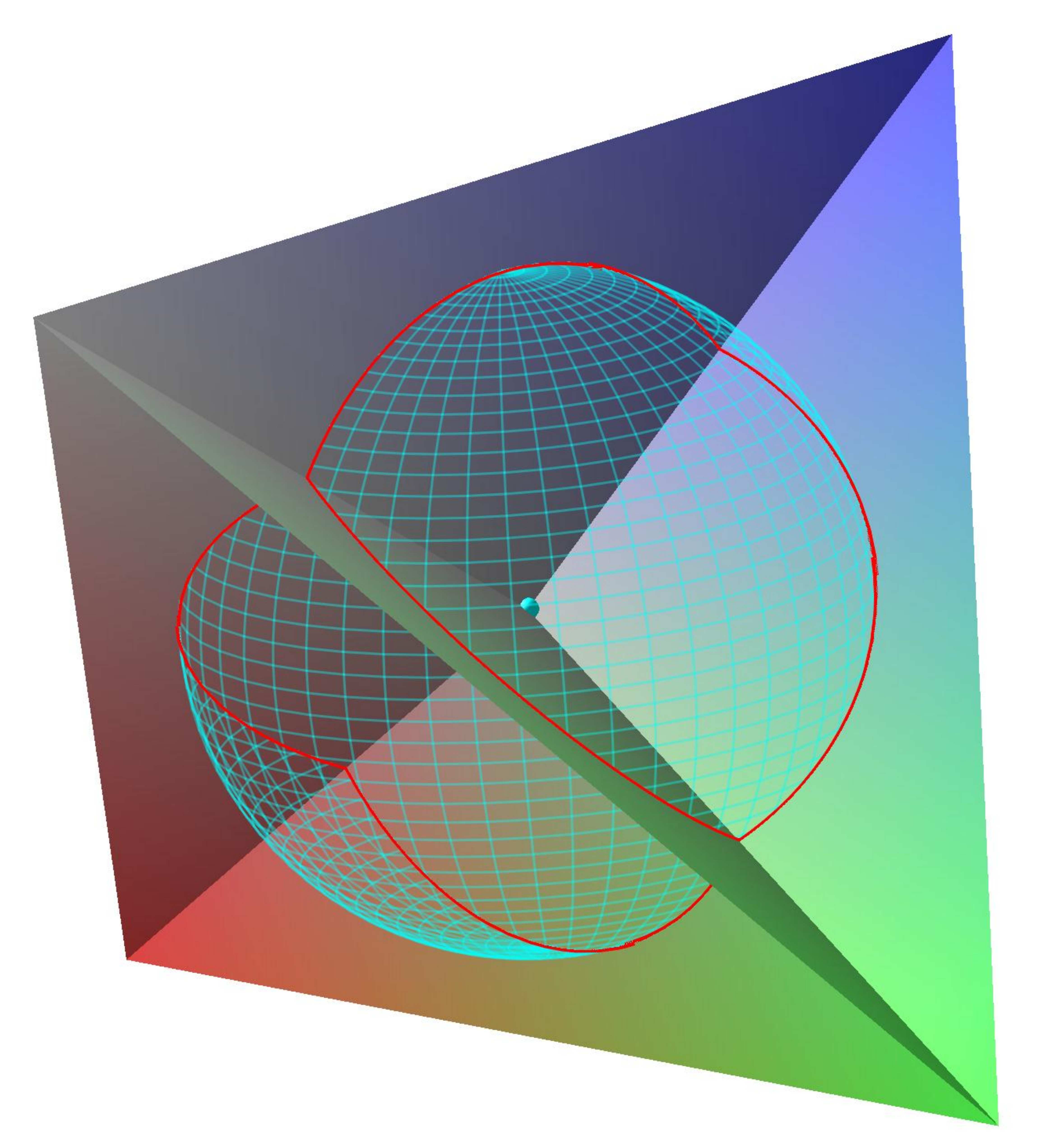} \qquad \qquad
\includegraphics[height=20mm]{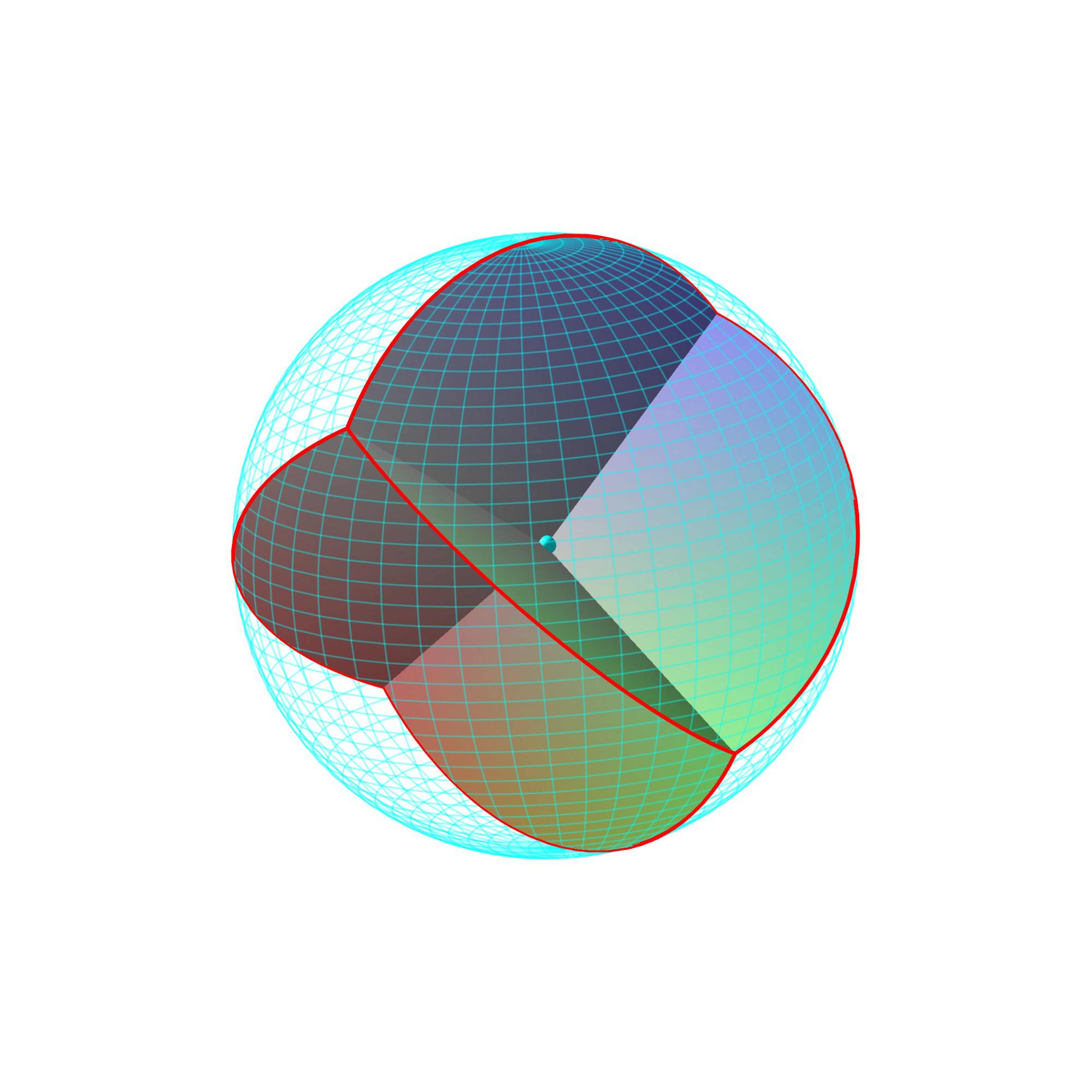} \qquad \qquad
\includegraphics[height=20mm]{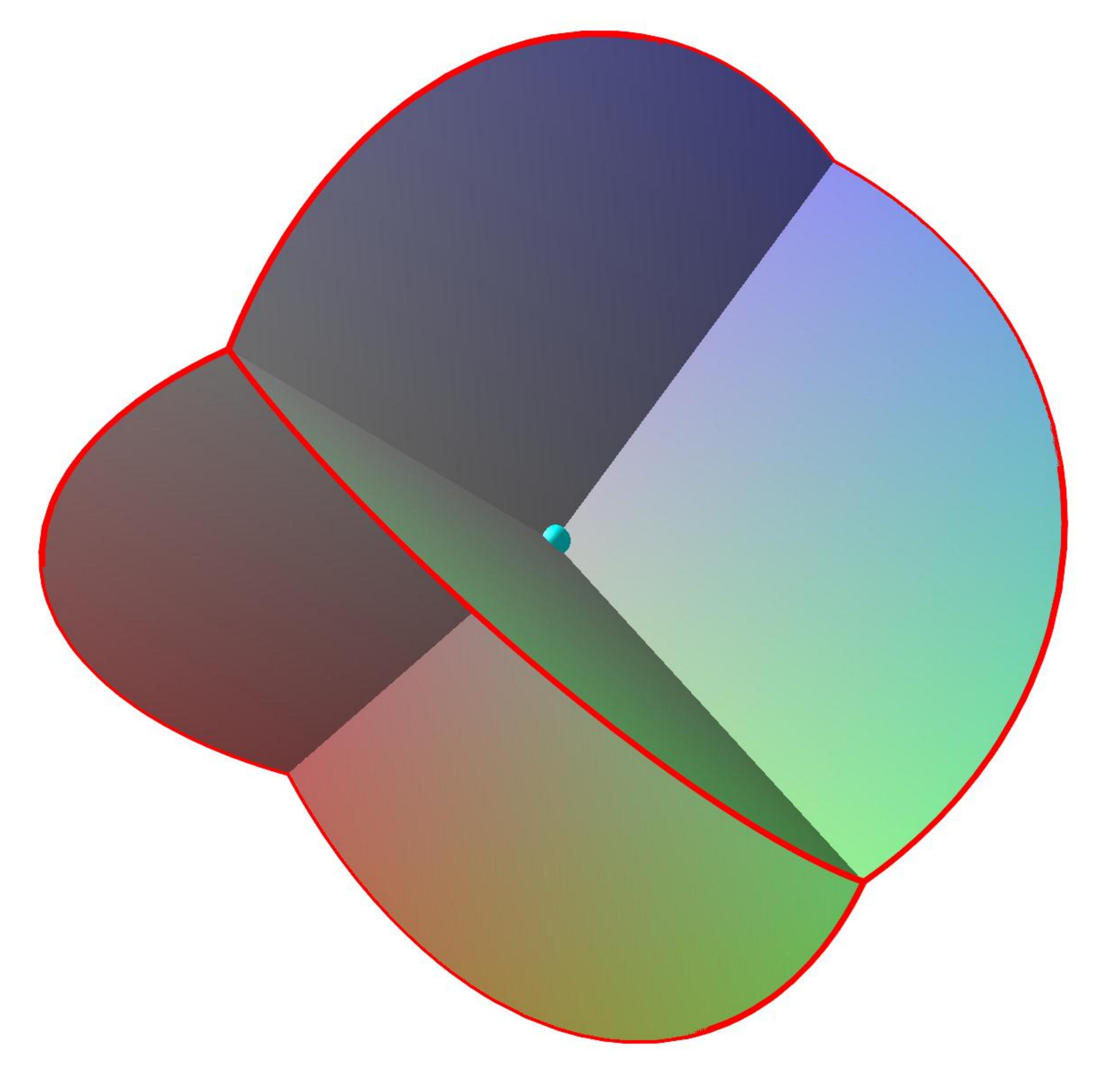}}
\begin{center}
\caption{\small{Cutting off spherical scoops -- extrinsic foam discs -- from  affine foams around an \emph{edge point} (top row) and a \emph{vertex point} (bottom row), respectively.}}\label{tetraball}
\end{center}
\end{figure}

We will assume throughout that the length of the mean curvature vector of the faces of the foam $F$ is bounded by a constant $h$ as follows:
\begin{equation} \label{eqMeanCond}
|H(x)| \leq h \quad \textrm{for all} \quad x \, \in F\quad .
\end{equation}


Our main objective is to determine a lower bound for the area of an extrinsic ball (a 'spherical scoop') of a foam centered at a given vertex point. An extrinsic disc of radius $R$ centered at a vertex point is the intersection of a ball of the Euclidean $3$-space of radius $R$ centered at the vertex point of the foam. In figure \ref{extball} we show a numerical simulation of an extrinsic disc of the Kelvin foam  obtained by using the Surface Evolver program.

\begin{figure}
\centerline{
\includegraphics[scale=0.27]{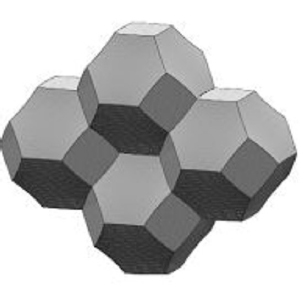} \qquad \includegraphics[scale=0.20]{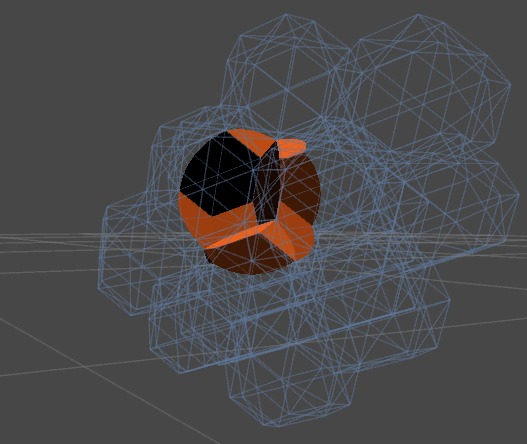}\qquad \includegraphics[scale=0.15]{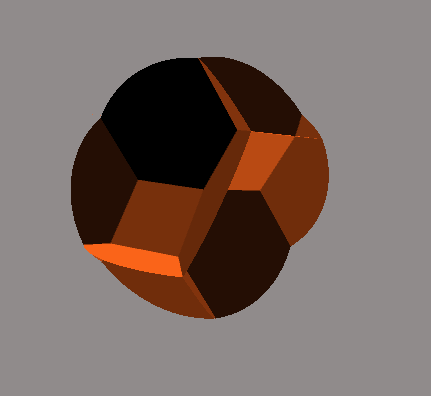}}
\begin{center}
\caption{\small{Piece of the Kelvin foam and an (enlarged) extrinsic disc of the Kelvin foam centered at a vertex point.}}\label{extball}
\end{center}
\end{figure}


\subsection{Main result. Area comparison} \label{secMain}

Let $F$ be a foam satisfying the mean curvature bound \eqref{eqMeanCond}. The main result of this paper states that for any point $o$ (which is not necessarily a vertex point) in $F$, and for every spherical scoop,  extrinsic foam disc,  $D_R(o)$ of radius $R$ centered at $o$, \emph{i.e.} the intersection $F \cap B_R(o)$, where $B_{R}(o)$ denotes the open Euclidean distance ball with radius $R$ and center $o$ in $\mathbb{R}^{3}$, a sharp lower bound for the area of the extrinsic disc can be obtained in terms of the radius, the upper bound for the norm of the mean curvature vector field and the type of the point $o$.

More precisely, we can state the following main theorem,

\begin{thmain} \label{thmainA}
Let $F$ be a foam properly immersed satisfying the mean curvature bound \eqref{eqMeanCond}. Then the area of the extrinsic foam disc $D_R(o)$ of radius $R$ centered at $o$ is bounded from below by
\begin{equation}\label{MainIneq}
\textrm{A}\left(D_R(o)\right)\geq \theta(o) \cdot e^{-2hR}\cdot  \pi R^2,
\end{equation}
where
\begin{equation}
\theta(o)=
\left\{
\begin{array}{lcl}
\theta_{v} = \frac{3}{\pi}\arccos(-1/3)&\textrm{if }&o\textrm{ is a vertex of } F\\ \\
\theta_{e} = \frac{3}{2}&\textrm{if }&o\textrm{ lies in an edge of } F\\ \\
\theta_{f} = 1&\textrm{if }&o\textrm{ lies in a face of } F \quad .
\end{array}
\right.
\end{equation}
Furthermore, equality in the  inequality (\ref{MainIneq}) is attained if and only if every face element of the extrinsic foam disc $D_R(o)$ is a piece of an affine  plane containing $o$.
\end{thmain}
\begin{remark}
This main theorem can be viewed as a concrete explication of
a general result of \cite{Allard}  with special focus on the particular structural data
that appear for Plateau foam varifolds in $\mathbb{R}^{3}$. In this paper we provide a proof using an adapted version of the divergence theorem for foams. Moreover, in Theorem \ref{teo5.2} below we extend the main theorem and obtain an even sharper global lower area bound for foams using a specific optimal measure on the vertex set of the respective foams, see definition \ref{def1}.
\end{remark}

\begin{simpexamples}
The simplest non-planar example of a foam consists of just one spherical bubble $F$  with $1$ face (of radius $\rho$), no edges and no vertices. See Figure \ref{figSimpExamp}. In this case, Theorem \ref{thmainA} says that for any point $o$ on the sphere, the area of the extrinsic $R$-disc centered at $o$ is:
\begin{equation} \label{eqSphere}
\area(F \cap B_{R}(o)) \geq \pi R^{2} \cdot e^{-2R/\rho} \quad,
\end{equation}
where we have used the optimal mean curvature bound $h = 1/\rho$ for $F$. The inequality \eqref{eqSphere} is easily seen to be true for all $R$ independent of the given $\rho$, since in this case we have:
\begin{equation}
\area(F \cap B_{R}(o)) = \left\{
  \begin{array}{ll}
    \pi\,R^{2} & \hbox{\textrm{for} $\quad R \leq 2\rho\,\,$;} \\ \\
    4\pi\,\rho^{2} & \hbox{\textrm{for} $\quad R \geq 2\rho \quad .$}
  \end{array}
\right.
\end{equation}

Another simple example $F$  with $1$ face, no edges and no vertices is the cylinder with radius $\rho$, which is also on display in Figure \ref{figSimpExamp}. The estimate of Theorem \ref{thmainA} is the same as before:
\begin{equation} \label{eqCyl}
\area(F \cap B_{R}(o)) \geq \pi R^{2} \cdot e^{-2R/\rho} \quad.
\end{equation}
Again this is easily verified; in particular for large $R \, >> \, \rho$ we have:
\begin{equation}
\area(F \cap B_{R}(o)) \, \dot{=} \, 2 \pi R \cdot \rho\, >> \,  \pi R^{2} e^{-2R/\rho} \quad .
\end{equation}

Finally, we should mention also the example supplied by the catenoid $F$. This is a minimal surface, so $h = 0$, and for any given point $o \in F$ and very large $R$ the intersection $F \cap B_{R}(o)$ is essentially two large discs of radius $R$ as indicated on the left in Figure \ref{figSimpExamp}. In consequence, the area estimate is thus:
\begin{equation}
\area(F \cap B_{R}(o)) \, \dot{=} \, 2 \pi R^{2} \, >  \,  \pi R^{2} \quad ,
\end{equation}
which again gives a rough verification of the theorem.
\end{simpexamples}

\begin{figure}
\centerline{
\includegraphics[width=23mm]{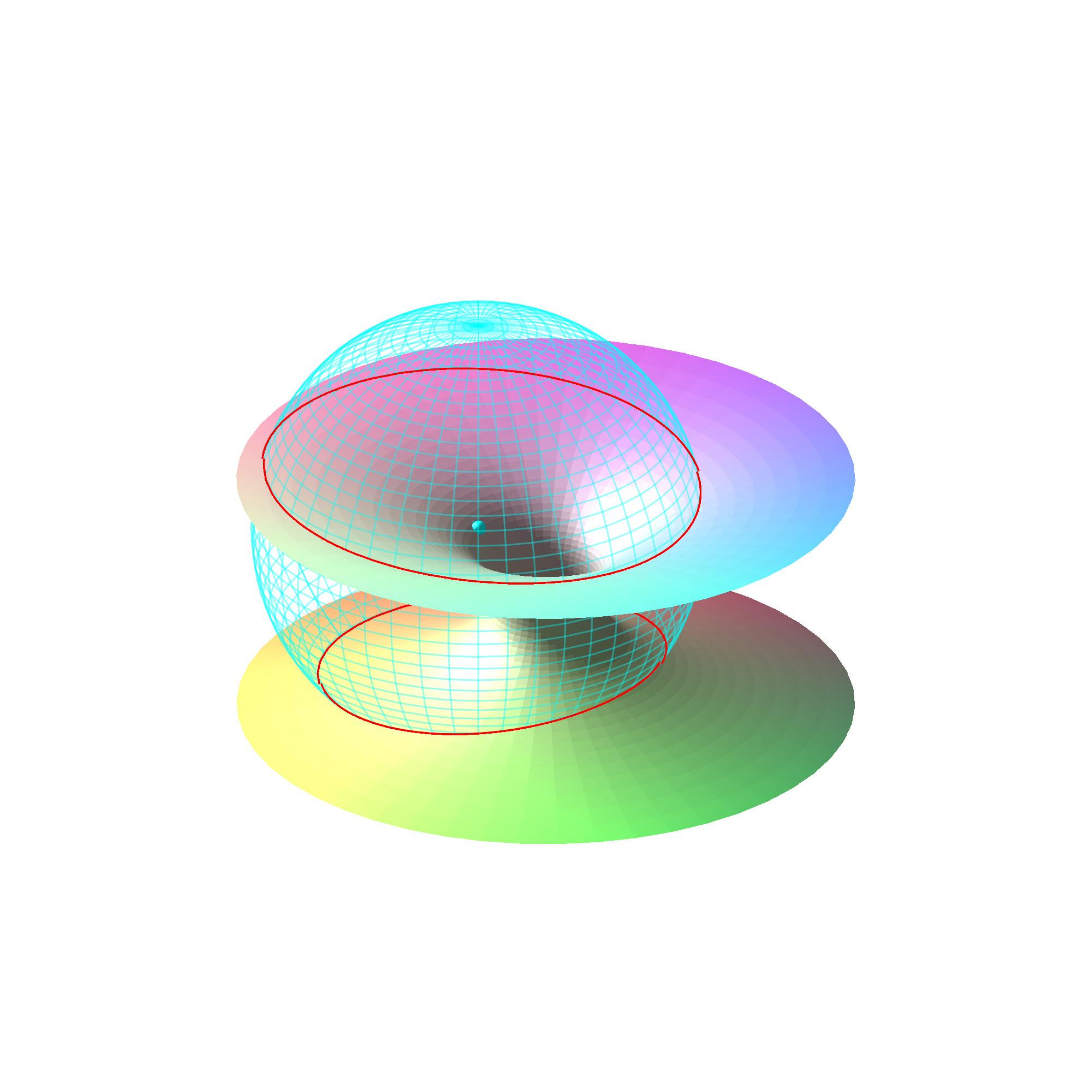} \includegraphics[width=23mm]{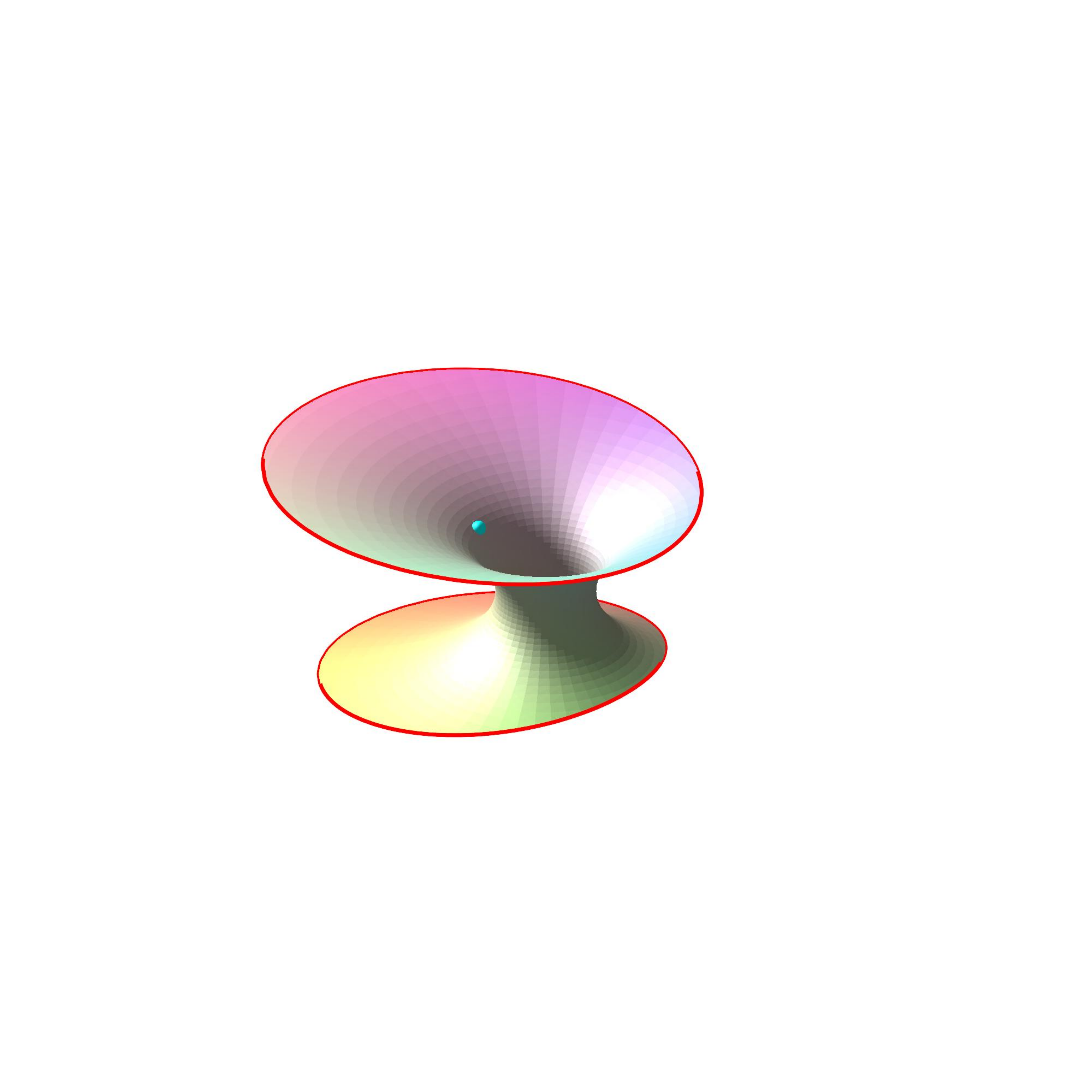}\,\,\,\includegraphics[width=23mm]{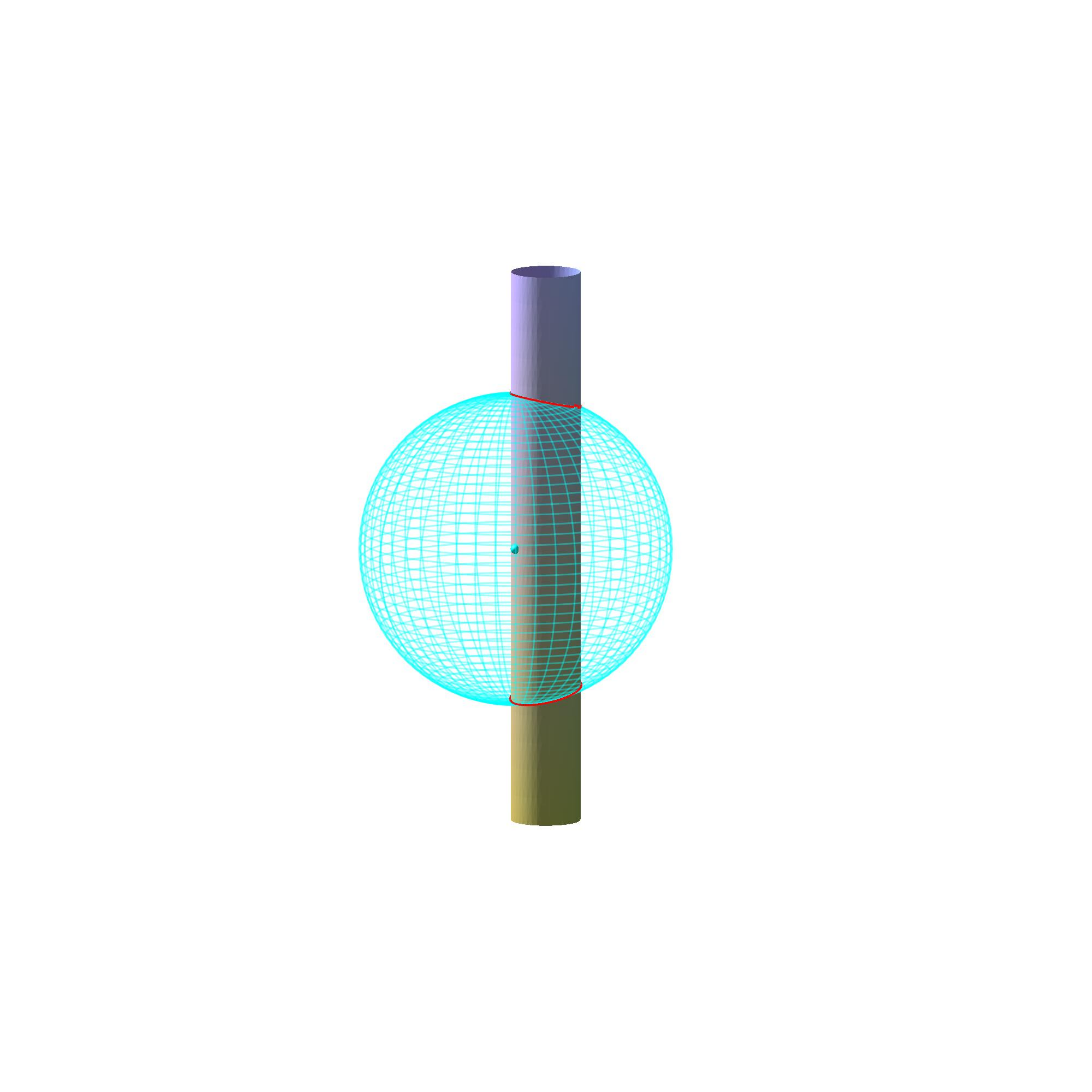}
\includegraphics[width=23mm]{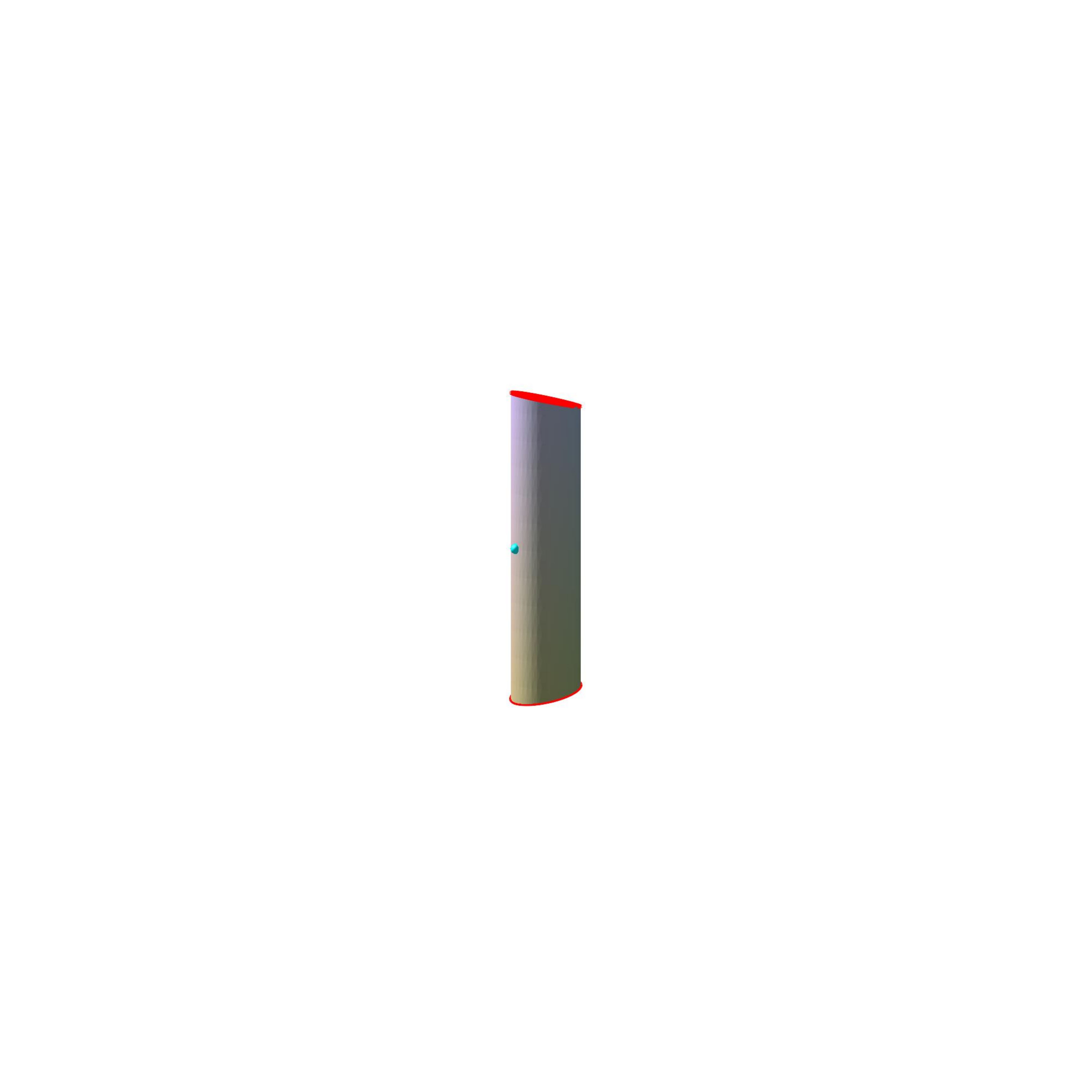}\,\,\,\includegraphics[width=23mm]{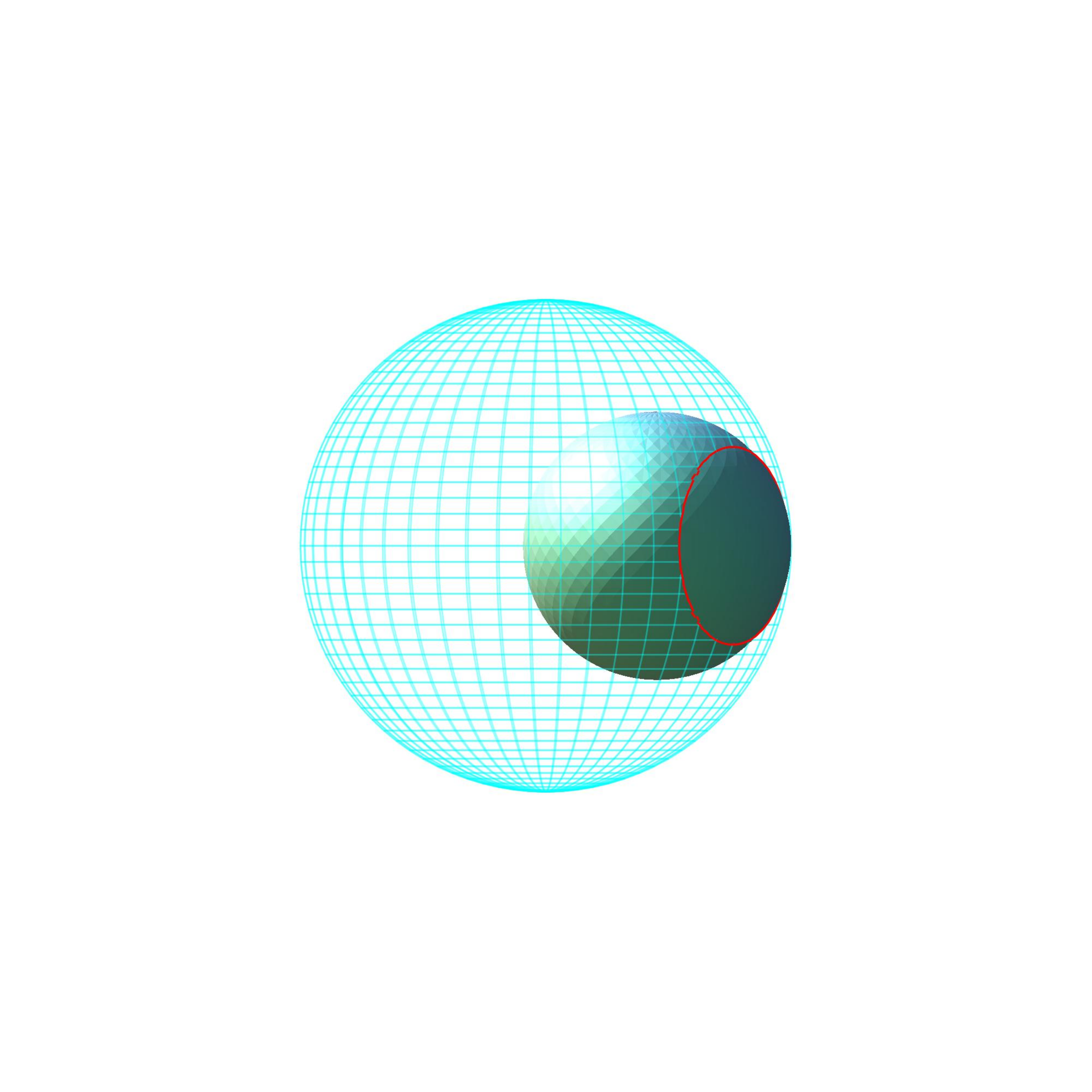}
\includegraphics[width=23mm]{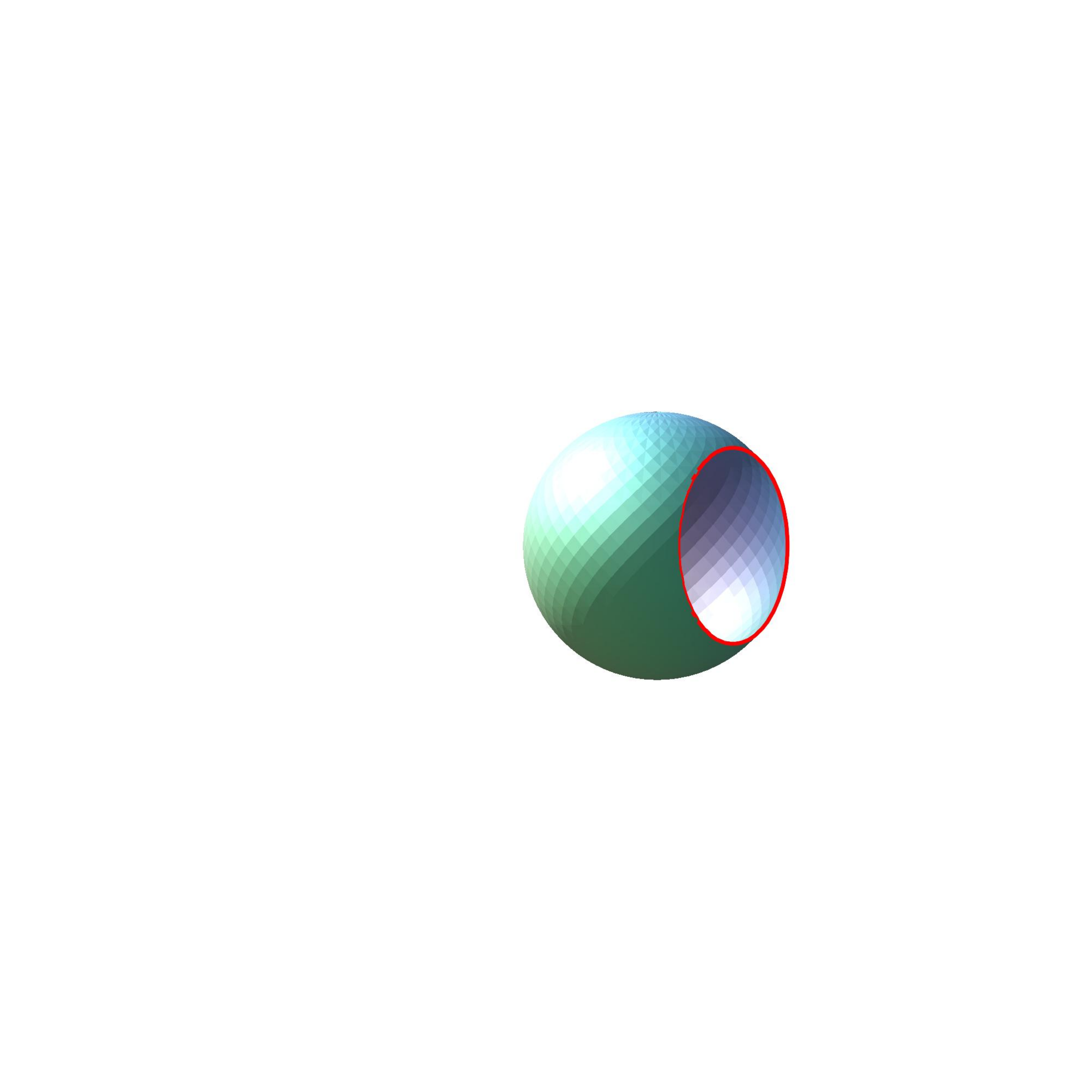}}
\caption{\small{From left to right: cutting off extrinsic foam discs from a catenoid, a cylinder and a sphere}}\label{figSimpExamp}
\end{figure}

As already alluded to, the main tool of the present paper is the statement and proof of an adapted version of the divergence theorem for foams,  theorem \ref{divtheofoams}.

\section{Outline of the Paper}
In section \ref{prelim} we shall introduce all required propositions and preliminaries, including the adapted version of the divergence theorem, needed to prove the main theorem of the paper  in section  \ref{proofMain}. The last part of the paper, section \ref{aplic}, is devoted to show some applications of the main theorem. These applications are:
\begin{enumerate}
\item  A lower bound for the area of a compact foam, \S\ref{aplic} \ref{secCompFoam}.
\item A recipe on how to use the main theorem to obtain a lower bound for the contribution of the area of a cell to the total area of a foam. Actually, in \S\ref{aplic} \ref{secKelvin}, we show how to obtain, via the main theorem, a lower bound for the area of a cell in the Kelvin foam.
\item A lower bound for the cost function in \S \ref{aplic}\ref{secCost} and a lower bound for the pressure in minimal foams in  \S \ref{aplic}\ref{secPreasure}.
\item Finally, in \S \ref{aplic}\ref{secEva}, we provide an algorithm to compute a lower bound for the area of a foam taking into account only the position of the vertices of the foam.
\end{enumerate}
\subsection*{Acknowledgment}We would like to thank Frank Morgan for reminding us about the paper \cite{Allard}, where such volume comparison formulas appear in the more general setting of varifolds with generalized mean curvatures.
\section{Preliminaries}\label{prelim}
By using an appropriate version of the divergence theorem on extrinsic discs of foams and the general {Co-area} formula given by geometric measure theory, we state the main comparison theorem for the area of an extrinsic foam discs.

In a  foam $F\subset \erre^3$ there are three different kinds of points: the points on the interior of the faces, the points on the interior of the edges, and the points on the vertices. The faces in a foam are smooth CMC surfaces (\emph{i.e.}, with constant norm of the mean curvature vector field). They are orientable and support individually a well defined normal vector field. We can consider that each face of the foam is embedded in $\erre^3$.

\subsection{Intrinsic and extrinsic distance function on a foam}

Given the inclusion map $i:F\to\erre^3$ for the foam to the Euclidean space $\erre^3$, we {will} say that a map $\gamma: I\to F$, from the interval $I\subset \erre$  is a \emph{piece-wise smooth curve segment of $F$} if the composed map with the inclusion $\widetilde \gamma:I\to\erre^3$, $\widetilde \gamma=i\circ \gamma$ is a piecewise smooth curve segment of $\erre^3$. Every piecewise smooth curve segment $\gamma$ has length ${\rm L}_F(\gamma)$ given by the length of the associated piecewise smooth curve segment $\widetilde\gamma$ in $\erre^3$, \emph{i.e.}, ${\rm L}_F(\gamma)={\rm L}(\widetilde\gamma)$.

Given two points, $p,q\in F$ we can now define the\emph{ intrinsic distance} from $p$ to $q$, denoted by ${\rm dist}_F(p,q)$, to be the infimum of ${\rm L}(\gamma)$ over all piece-wise smooth curve segments $\gamma$ from $p$ to $q$. With this definition, $(F,{\rm dist}_F)$ becomes a metric space.

If we choose a point $o\in\erre^3$, we can define the \emph{extrinsic distance function} $r_o:F\to\erre$, by
$$
r_o(x)={\rm dist}_{\erre^3}(o,x).
$$
Hence, for any two points $p,q\in F$,
$$
\vert r_o(p)-r_o(q)\vert\leq{\rm dist}_{\erre^3}(p,q)\leq{\rm}{\rm dist}_{F}(p,q).
$$
That means that the extrinsic distance function is a $1$-Lipschitz function on the metric space $(F,{\rm dist}_F)$. Observe that, since each face $\{F_i\}$ of $F$ is a smooth surface of $\erre^3$, then the extrinsic distance function is a $C^\infty$ function on $F_i\setminus \{o\}$. Moreover for any $v\in T_pF_i$, denoting $\rho_o(x)={\rm dist}_{\erre^3}(o,x)$,
\begin{equation}
dr_o(v_i)=\langle \nabla^Fr_o,v_i\rangle=d\rho_o(v_i)=\langle \nabla^{\erre^3} \rho_o,v_i\rangle\leq \Vert v_i\Vert,
\end{equation}
which implies
\begin{equation}
\Vert \nabla^{F_i}r_o\Vert\leq 1.
\end{equation}

By using  Sard's theorem, the set of critical values of $r_o:F_i\setminus\{o\}\to \erre_+$ has zero Lebesgue measure in $\erre^+$.  The set of sublevel set of the extrinsic distance is precisely the extrinsic disc, namely
\begin{equation}
D_R(o)=\{x\in F\,:\,r_o(x)<R\}=B_R(o)\cap F.
\end{equation}
The boundary of the extrinsic disc $D_R$ will be denoted by $\partial D_R$, \emph{i.e.},
\begin{equation}
\partial D_R(o)=\{x\in F\,:\,r_o(x)=R\}=S_R(o)\cap F.
\end{equation}
Similarly, for $\rho<R$ the extrinsic annulus is
\begin{equation}
A_{\rho,R}(o):=\{x\in F\,:\, \rho\leq r_o(x)\leq R\}
\end{equation}
Hence, by using  the Co-area formula (see \cite{Sakai} for instance), we can state
\begin{proposition}\label{coarea-apply}
Let $D_R(o)$ be an extrinsic disc of a foam $F$. Suppose that $R$ is a regular value of $r_o:F_i\setminus\{o\}\to\erre$ for every face $F_i$ of $F$, and suppose moreover the $\partial D_R$ meets transversally every edge of $F$. Then,
\begin{equation}
\frac{d}{dt}{\rm A}(D_t(o))\vert_{t=R}=\sum_{i}\int_{F_i\cap\partial D_R}\frac{dL_i}{\Vert \nabla^{F_i} r_o\Vert}\quad.
\end{equation}
\end{proposition}

\subsection{Divergence theorem on foams} \label{secDivergence}
Let us now, without loss of generality, center a point $o\in F$ of the foam $F$ in the origin of $\erre^3$ and consider the extrinsic disc $D_R(0)=F\cap B_R(0)$ of radius $R$ and center $0\in \erre^3$ given by the intersection of the foam $F$ with the ball $B_R(0)$ in $\erre^3$ of radius $R$ and centered at $0\in \erre^3$. We impose moreover that $R$ is such that every edge meets $\partial D_R$ transversally. Such extrinsic disc $D_R(0)$ centered at  $0\in \erre^3$ is not a smooth surface but is composed by a finite number $f_R$ of faces (which are in fact individually smooth surfaces), a finite number of edges $e_R$ and a finite number of vertices $v_R$. Namely,
\begin{equation}
D_R(0)=\{F_i\}_{i=1}^{f_R}\cup\{E_i\}_{i=1}^{e_R}\cup\{V_i\}_{i=1}^{v_R}
\end{equation}
where $\{F_i\}_{i=1}^{f_R},\{E_i\}_{i=1}^{e_R},\{V_i\}_{i=1}^{v_R}$ are the sets of faces, edges and vertices respectively. Let us denote by
\begin{equation}
\begin{aligned}
\mathcal{F}_R:=&\{F_i\}_{i=1}^{f_R}\\
\mathcal{E}_R:=&\{E_i\}_{i=1}^{e_R}\\
\mathcal{V}_R:=&\{V_i\}_{i=1}^{v_R}
\end{aligned}
\end{equation}

Let us denote by $n_i$ a unit normal vector field to the face $F_i$.
Hence, given a smooth vector field $X:B_R(0)\to T\erre^3$ in $B_R(0)\subset \erre^3$ we can obtain a vector field $X_{F_i}$ tangent and smooth on each face $F_i$ by using the normal vector field $n_i$, in such a way that

\begin{equation}
X_{F_i}:=X-\langle X,n_i\rangle n_i.
\end{equation}
Denote by $X_F$ the application $X_F: F\setminus \mathcal{E}_R\to T\erre^3$ given by
\begin{equation}
X_F(x)=X_{F_i}(x)
\end{equation}
where $F_i$ is the face such that $x \in F_i$. With these definitions we can state the following divergence theorem on foams

\begin{theorem}[Divergence theorem on foams]\label{divtheofoams}
Let $F\subset \erre^3$ be a foam in $\erre^3$. Let $D_R(0)$ be the extrinsic disc given by $D_R(0)=F\cap B_R(0)$, then for any vector field $X: B_R\to T\erre^3$

\begin{equation}
\int_{D_R(0)}\textrm{div} X_F dA=\int_{\partial D_R(0)}\langle X,\nu_F\rangle dL
\end{equation}
where
\begin{equation}
\int_{\partial D_R(0)}\langle X,\nu_F\rangle dL:=\sum_{i=1}^{f_R}\int_{F_i\cap \partial B_R(0)}\langle X,\nu_i^R\rangle dL_i.
\end{equation}
\end{theorem}

\begin{proof}
We can refer the integral of the divergence of the tangential component $X_F$ of vector field $X \in \mathfrak{X}(\erre^3)$  to the sum of the divergence in each face. Namely,
\begin{equation}
\int_{D_R(0)}\textrm{div} X_F dA:=\sum_{i=1}^{f_R}\int_{F_i}\textrm{div}X_{F_i}dA_i
\end{equation}
Observe that on each face $F_i$ we can have two kinds of boundary components of $\partial F_i$,
\begin{equation}
\partial F_i=\left(F_i\cap \partial B_R\left(0\right)\right) \cup \left(F_i\cap \mathcal{E}_R\right)
\end{equation}
Accordingly, let us denote by $\nu_i^R$ the outward unit normal vector field on $F_i\cap \partial B_R(0)$ with $\nu_i^R=0$ if $F_i\cap \partial B_R(0)=\emptyset$ and by $\nu_{i,j}$ the outward unit normal vector field on $F\cap e_j$ with $\nu_{i,j}=0$ if $F\cap e_j=\emptyset$.
Observe that by the structure of the foam and the three-faces-one-edge with equal angles we have
\begin{equation}
\sum_{i=1}^{f_R}\nu_{i,j}=0,
\end{equation}
for any $j\in \{1,\cdots,e_R\}$.
Hence, by using the divergence theorem on each face (see theorem 14.34 of \cite{LeeBook} for instance) we get:

\begin{eqnarray}
\int_{D_R(0)}\textrm{div} X_F dA & = &\sum_{i=1}^{f_R}\int_{F_i}\textrm{div}X_{F_i}dA_i\nonumber\\
&=&\sum_{i=1}^{f_R}\left(\int_{F_i\cap \partial B_R(0)}\langle X_{F_i},\nu_i^R\rangle dL_i +\sum_{j=1}^{e_R}\int_{F_i\cap e_j}\langle X_{F_i},\nu_{i,j}\rangle dL_{i,j}\right)\nonumber\\
&=&\sum_{i=1}^{f_R}\left(\int_{F_i\cap \partial B_R(0)}\langle X,\nu_i^R\rangle dL_i +\sum_{j=1}^{e_R}\int_{e_j}\langle X,\nu_{i,j}\rangle dL_{j}\right)\\
&=&\sum_{i=1}^{f_R}\int_{F_i\cap \partial B_R(0)}\langle X,\nu_i^R\rangle dL_i +\sum_{j=1}^{e_R}\int_{e_j}\langle X,\sum_{i=1}^{f_R}\nu_{i,j}\rangle dL_{j}\nonumber\\
&=&\sum_{i=1}^{f_R}\int_{F_i\cap \partial B_R(0)}\langle X,\nu_i^R\rangle dL_i.\nonumber
\end{eqnarray}\end{proof}

\subsection{Laplacian of the coordinate functions of $\erre^3$ on a CMC surface}
Recall that given a smooth surface $S\subset \erre^3$ in the Euclidean space $\erre^3$, the covariant derivatives on $S$ and $\erre^3$ are related by the Gauss formula by
\begin{equation}
\nabla_{\overline X}^{\erre^3}\overline{Y}=\nabla_X^{S}Y+\alpha(X,Y)
\end{equation}
for any two vector fields $X,Y\in\mathfrak{X}(S)$ and any two extensions $\overline X,\overline Y\in\mathfrak{X}(\erre^3)$. the term $\alpha$ is the second fundamental form of $S$. Similarly, the Hessian operators $\hess^S$ and $\hess^{\erre^3}$ are related by
\begin{proposition}\label{hess-comp}Let $f:\erre^3\to\erre$ be a smooth function and let us denote also by $f:S\to\erre$ the restriction of the function to the smooth surface $S\subset \erre^3$, then
\begin{equation}
\hess^Sf(X,Y)=\hess^{\erre^3}f(\overline X,\overline Y)+\langle \alpha(X,Y),\nabla^{\erre^3}f\rangle
\end{equation}
for any two vector fields $X,Y\in\mathfrak{X}(S)$ and any two extensions $\overline X,\overline Y\in\mathfrak{X}(\erre^3)$.
\end{proposition}
\begin{proof}
\begin{equation}
\begin{aligned}
\hess^Sf(X,Y)=&\langle\nabla^S_X\nabla^Sf,Y\rangle=X(\langle\nabla^Sf,Y\rangle)-\langle\nabla^Sf,\nabla_XY\rangle\\
=& X(\langle\nabla^{\erre^3}f,Y\rangle)-\langle\nabla^{\erre^3}f,\nabla^S_XY\rangle\\
=& X(\langle\nabla^{\erre^3}f,Y\rangle)-\langle\nabla^{\erre^3}f,\nabla^{\erre^3}_XY\rangle+\langle\nabla^{\erre^3}f,\alpha(X,Y)\rangle\\
=&\hess^{\erre^3}f(\overline X,\overline Y)+\langle \alpha(X,Y),\nabla^{\erre^3}f\rangle.
\end{aligned}
\end{equation}\end{proof}
\noindent The above proposition leads us to
\begin{proposition}\label{laplacian-coordinates}Let $\varphi:S\to\erre^3$ be an immersed surface in $\erre^3$.
Let us denote by $\{x,y,z\}$ the coordinate functions in $\erre^3$ and their restrictions to $S$.  Let
\begin{equation}
H=\frac{1}{2}{\rm trace}_g(\alpha)=(H_1,H_2,H_3)
\end{equation}
be the mean curvature vector field of $S$. Then
\begin{equation}
\begin{aligned}
\triangle^Sx&=2H_1,\\
\triangle^Sy&=2H_2,\\
\triangle^Sz&=2H_3.
\end{aligned}
\end{equation}
\end{proposition}
\begin{proof}
Let us consider the following vector $a=(a_1,a_2,a_3)$ and the following function
\begin{equation}
f_a:\erre^3\to\erre,\quad (x,y,z)\to f_a(x,y,z)=a_1x+a_2y+a_3z.
\end{equation}
One can easily check that
\begin{equation}
\nabla^{\erre^3}f_a=a,
\end{equation}
and hence
\begin{equation}
\hess^{\erre^3}f_a(\overline X,\overline Y)=0.\quad \forall \overline X,\overline Y\in\mathfrak{X}(\erre^3).
\end{equation}
Using therefore proposition \ref{hess-comp},
\begin{equation}
\hess^{S}f_a(X,Y)=\langle \nabla^{\erre^3}f_a,\alpha(X,Y)\rangle=\langle a, \alpha(X,Y)\rangle
\end{equation}
thus, for any orthonormal basis $\{e_1,e_2\}$ of $T_pS$ at $p\in S$,
\begin{equation}
\triangle^Sf_a=\sum_{i=1}^2\hess^{S}f_a(e_i,e_i)=\sum_{i=1}^2\langle a, \alpha(e_i,e_i)\rangle=\langle a, \sum_{i=1}^2\alpha(e_i,e_i)\rangle=2\langle a,H\rangle.
\end{equation}
Finally, the proposition follows for the particular cases $a=(1,0,0)$ or $a=(0,1,0)$ or $a=(0,0,1)$.
\end{proof}

\subsection{Density of a foam}
Denoting by $\theta(o)$ the density of the point $o$, \emph{i.e.} (see also definition 2.5 of \cite{M}),
$$\theta(o)=\lim_{s\to 0^+}\frac{\textrm{A}(D_s(o))}{\pi s^2}$$ we can state
\begin{proposition}\label{density}Let $F\subset \erre^3$ be a foam. Then for any $o\in F$
\begin{equation}
\theta(o)=
\left\{
\begin{array}{lcl}
\frac{3}{\pi}\arccos(-1/3)&\textrm{if }&o\textrm{ is a vertex of } F\\
\frac{3}{2}&\textrm{if }&o\textrm{ lies in the interior of an edge of } F\\
1&\textrm{if }&o\textrm{ lies in the interior of a face of } F\\
\end{array}
\right.
\end{equation}
\end{proposition}
\section{Proof of the main theorem}\label{proofMain}
In proposition \ref{regular-lower} below we will be able to obtain a lower bound for the area of an extrinsic disc of  foam. First we need the following
\begin{proposition}\label{Upper}
Let $D_R(o)$ be an extrinsic disc of a foam $F$. Suppose that $R$ is a regular value of $r_o:F_i\setminus\{o\}\to\erre$ for every face $F_i$ of $F$, and suppose moreover the $\partial D_R$ meets transversally every edge of $F$. Then,
\begin{equation}
{\rm A}(D_R)\left(1-hR\right) \leq \frac{R}{2}\sum_{i=1}^{f_R}\int_{F_i\cap \partial D_R}\vert \nabla^{F_i}r_o\vert dL_i.
\end{equation}
\end{proposition}
\begin{proof}
Without loss of generality suppose that $o=0\in \erre^3$.
We are using the divergence theorem on foams (theorem \ref{divtheofoams}) with the vector field in $B_R(o)$ given by
\begin{equation}
X=\nabla^{\erre^3}\phi,
\end{equation}
where $\nabla^{\erre^3}$ is the gradient of $\erre^3$ and
\begin{equation}
\phi=\frac{1}{4}\left(x^2+y^2+z^2\right).
\end{equation}
$x,y,z$ being the coordinate functions in $\erre^3$.
Hence, on every face $F_i\subset F$
\begin{equation}
\begin{aligned}
X_{F_i}=&\nabla^{\erre^3}\phi-\langle \nabla^{\erre^3}\phi, n_i\rangle n_i=\nabla^{F_i}\phi_{\vert F_i}\\
=&\frac{1}{2}\left(x\nabla^{F_i}x+y\nabla^{F_i}y+z\nabla^{F_i}z\right).
\end{aligned}
\end{equation}
Thus,
\begin{equation}
\begin{aligned}
\textrm{div}X_{F_i} = &\frac{1}{2}\left(\langle \nabla^{F_i} x,\nabla^{F_i}x\rangle+\langle \nabla^{F_i} y,\nabla^{F_i}y\rangle +\langle \nabla^{F_i} z,\nabla^{F_i}z\rangle\right)\\
& +\frac{1}{2}\left(x\triangle^{F_i}x+y\triangle^{F_i}y+z\triangle^{F_i}z\right).
\end{aligned}
\end{equation}
Denoting by $H$ the mean curvature of the face $F_i$ and applying proposition \ref{laplacian-coordinates}
\begin{eqnarray}
\textrm{div}X_{F_i}&=&\frac{1}{2}\left(\vert \nabla^{F_i} x\vert^2+\vert \nabla^{F_i} y\vert^2 +\vert \nabla^{F_i} z\vert^2\right)+\langle \vec{r}, H\rangle\nonumber\\
&=&\frac{1}{2}\left(\vert \nabla^{\erre^3} x\vert^2-\langle\nabla^{\erre^3} x,n\rangle^2 + \vert \nabla^{\erre^3} y\vert^2\right.\\
&& \left.\quad-\langle\nabla^{\erre^3} y,n\rangle^2 + \vert \nabla^{\erre^3} z\vert^2-\langle\nabla^{\erre^3} z,n\rangle^2 \right)+\langle \vec{r}, H\rangle\nonumber\\
&=&\frac{1}{2}\left(3-\vert n\vert^2\right)+\langle \vec{r}, H\rangle=1+\langle \vec{r}, H\rangle\nonumber.
\end{eqnarray}
Here, $\vec{r}=(x,y,z)$. Then
\begin{equation}
\textrm{div}X_{F_i}\geq 1-\vert H\vert R\geq 1-hR
\end{equation}

Using theorem \ref{divtheofoams} and denoting $\partial D_R^i=F_i\cap \partial D_R$, therefore,
\begin{eqnarray}
\textrm{A}(D_R(o))(1-hR)&\leq&\sum_{i=1}^{f_R}\int_{F_i}{\rm div}X_{F_i}dA_i=\sum_{i=1}^{f_R}\int_{\partial D_R^i}\langle X,\nu_i^R\rangle dL_i\nonumber\\
&=&\sum_{i=1}^{f_R}\int_{\partial D_R^i}\langle \nabla^{\erre^3}\phi,\nu_i^R\rangle dL_i= \sum_{i=1}^{f_R}\int_{\partial D_R^i}\frac{r_o}{2}\Vert \nabla^{F_i}r_o\Vert dL_i\nonumber
\end{eqnarray}
\end{proof}

\begin{proposition}\label{regular-lower}
Let $F$ be a foam. For any $R>0$ denote by $h$ the maximum of the norm of the mean curvature vector field $n$ the faces of $D_R$, \emph{i.e.},
$$
h=\max_{x\in D_R}\Vert\vec H(x)\Vert.
$$
Then for any $\rho>0$ such that $\rho<R$
\begin{equation}\label{equsharp}
\frac{{\rm A}\left(D_R\right)}{R^2}e^{2hR}\geq\left( e^{\frac{1}{{\rm A}(D_R)}\underset{A_{\rho,R}\cap\mathcal {R}}{\int}\left(1-\Vert\nabla r_o\Vert^2\right)dV} \right)\frac{{\rm A}\left(D_\rho\right)}{\rho^2}e^{2h\rho}.
\end{equation}
\end{proposition}
\begin{proof}
Let us denote
$$
\mathcal{R}:=\left\{\begin{aligned}x\in F\,\vert\, & r_o(x)\text{ is a regular value and }\\ & \partial D_{r_0(x)} \text{ meets the edges of } F \text{ tranversally}\end{aligned}\right\}
$$
By using Co-area formula, proposition \ref{coarea-apply}, and proposition \ref{Upper}
\begin{eqnarray}
\frac{1}{{\rm A}(D_R)}\underset{A_{\rho,R}\cap\mathcal {R}}{\int}\left(1-\Vert\nabla r_o\Vert^2\right)dV=&\frac{1}{{\rm A}(D_R)}\int_\rho^R\left(\int_{\partial D_t}\frac{1-\Vert\nabla r_o\Vert^2}{\Vert \nabla r_o\Vert}d{\rm A}_t\right)dt\nonumber\\
\leq & \int_\rho^R\frac{1}{{\rm A}(D_t)}\left(\int_{\partial D_t}\left(\frac{1}{\Vert \nabla r_o\Vert}-\Vert\nabla r_o\Vert\right)d{\rm A}_t\right)dt\\
\leq & \int_\rho^R\frac{1}{{\rm A}(D_t)}\left(\frac{d}{dt}{\rm A}(D_t)-\frac{2}{t}{\rm A}(D_t)\left(1-ht\right)\right)dt\nonumber\\
=& \int_\rho^R\left(\frac{d}{dt}\log\left({\rm A}(D_t\right))-\frac{d}{dt}\left(\log\left(t^2\right)-2ht\right)\right)dt\nonumber
\end{eqnarray}
Taking into account that the function $t\to {\rm A}(D_t)$ is $C^\infty$ almost everywhere in $[\rho, R]$ and non-decreasind then,
\begin{equation}\label{eq3.9}
\log\left(\frac{\frac{{\rm A}\left(D_R\right)}{R^2}e^{2hR}}{\frac{{\rm A}\left(D_\rho\right)}{\rho^2}e^{2h\rho}}\right)\geq \frac{1}{{\rm A}(D_R)}\int_{A_{\rho,R}\cap\mathcal {R}}\left(1-\Vert\nabla r_o\Vert^2\right)dV\geq 0.
\end{equation}
Hence the proposition follows.
\end{proof}
By using the above proposition
$$
{\rm A}\left(D_R\right)\geq\left( e^{\frac{1}{{\rm A}(D_R)}\underset{A_{\rho,R}\cap\mathcal {R}}{\int}\left(1-\Vert\nabla r_o\Vert^2\right)dV} \right)\frac{{\rm A}\left(D_\rho\right)}{\pi\rho^2}e^{2h\rho}\cdot e^{-2hR}\pi R^2.
$$
Thence, letting $\rho$ tend to $0$ and using proposition \ref{density} we obtain
\begin{equation}\label{sharpequ2}
{\rm A}\left(D_R\right)\geq\left( e^{\frac{1}{{\rm A}(D_R)}\underset{D_R(o)\cap\mathcal {R}}{\int}\left(1-\Vert\nabla r_o\Vert^2\right)dV} \right)\theta(o)\cdot e^{-2hR}\pi R^2.
\end{equation}
this inequality leads us to the statement and proof of the main theorem:
\begin{thmain}
Let $F$ be a foam properly immersed satisfying the mean curvature bound \eqref{eqMeanCond}. Then the area of the extrinsic disc $D_R(o)$ of radius $R$ centered at $o$ is bounded from below by
\begin{equation}\label{MainIneq2}
\textrm{A}\left(D_R(o)\right)\geq \theta_o \cdot e^{-2hR}\cdot  \pi R^2,
\end{equation}
where
\begin{equation*}
\theta(o)=
\left\{
\begin{array}{lcl}
\theta_{v} = \frac{3}{\pi}\arccos(-1/3)&\textrm{if }&o\textrm{ is a vertex of } F\\ \\
\theta_{e} = \frac{3}{2}&\textrm{if }&o\textrm{ lies in theinterior of an edge of } F\\ \\
\theta_{f} = 1&\textrm{if }&o\textrm{ lies in the interior of a face of } F \quad .
\end{array}
\right.
\end{equation*}
Furthermore, equality in the  inequality (\ref{MainIneq2}) is attained if and only if every face of the extrinsic foam disc $D_R(o)$ is a piece of an affine  plane containing $o$.
\end{thmain}
\begin{proof}
Observe that inequality (\ref{MainIneq2}) follows from inequality (\ref{sharpequ2}). Moreover equality in inequality   (\ref{MainIneq2}) implies equality in (\ref{sharpequ2}), therefore
$$
\Vert\nabla r_o\Vert^2=1
$$
for every point in $D_R(o)\cap\mathcal {R}$. Then,  every face of the extrinsic disc $D_R(o)$ is a piece of an affine  plane containing $o$.
\end{proof}
\section{Applications}\label{aplic}
\subsection{Compact Foams}\label{secCompFoam}
We apply our results to the setting of  compact foams:
\begin{theorem}
Let $F\subset \mathbb{R}^3$ be a compact foam. Then given a point $o\in F$,
There exists $R_{\rm max}<\infty$ such that
$
F=D_{R_{\rm max}}(o),
$
there exists $h<\infty$ such that
$
\Vert \vec{H}\Vert(x)\leq h\quad \textrm{for any }\quad x\in F.
$
The maximum radius $R_{\rm max}$ and the supremum of the norm of the mean curvature vector field are related by
$$
R_{\rm max}\geq \frac{1}{h}
$$
Moreover, the foam has finite area ${\rm A}(F)<\infty$ and is bounded from below by
$$
{\rm A}(F)\geq \frac{\theta(o)}{e^2}\frac{\pi}{h^2}\cdot
$$
\end{theorem}
\begin{proof}
By proposition \ref{Upper} we know that
$$
{\rm A}(D_R) (1-hR)\leq {\rm L}(\partial D_R)
$$
where here ${\rm L}(\partial D_R)$ denotes the length of the boundary $\partial D_R$.  Let now $\{R_i\}$ be a sequence for regular values of the extrinsic distance converging to $R_{\rm max}$ then
$$
{\rm A}(D_{R_{\rm max}}) (1-hR_{\rm max})=\lim_{i\to\infty } {\rm A}(D_{R_i}) (1-hR_i)\leq {\rm L}(\partial D_{R_i})=0.
$$
Hence we obtain
$$
R_{\rm max}\geq \frac{1}{h}
$$
Finally by using the main theorem,
$$
{\rm A}(F)\geq {\rm A}(D_{\rm max})\geq {\rm A}(D_{1/h})\geq  \frac{\theta(o)}{e^2}\frac{\pi}{h^2}.
$$
\end{proof}

\begin{simpexamples}
The simplest example of a compact foam is a sphere $S_R$ of radius $R$ in $\mathbb{R}^3$. From any base point on the sphere we have  $R_{\rm max }=2R$ so that
$$
R_{\rm max}=2R=\frac{2}{h}\geq \frac{1}{h} ,
$$
which is in complete agreement with the inequality $R_{\rm max}\geq \frac{1}{h}$. On the other hand, since
$$
{\rm A}(S_R)=4\pi R^2=\frac{4\pi}{h^2} \geq \frac{\theta(o)}{e^2}\frac{\pi}{h^2}=\frac{\pi}{e^2h^2}.
$$
Another simple example of a compact foam is a double bubble (see figure \ref{doubleBoubble}). The maximum of the mean curvature is
$$
h=\frac{1}{r_2}.
$$
Hence, by the above corollary.
$$
R_{\rm max}\geq r_2
$$
If we choose the center in the point $P$, it is easy to check that $R_{\rm max}=2r_1\geq r_2$. The lower bound for the area given by the above corollary is
$$
A(F)\geq \frac{3}{2e^2}\pi r_2^2.
$$
\end{simpexamples}

\begin{figure}
\centerline{
\includegraphics[scale=0.5]{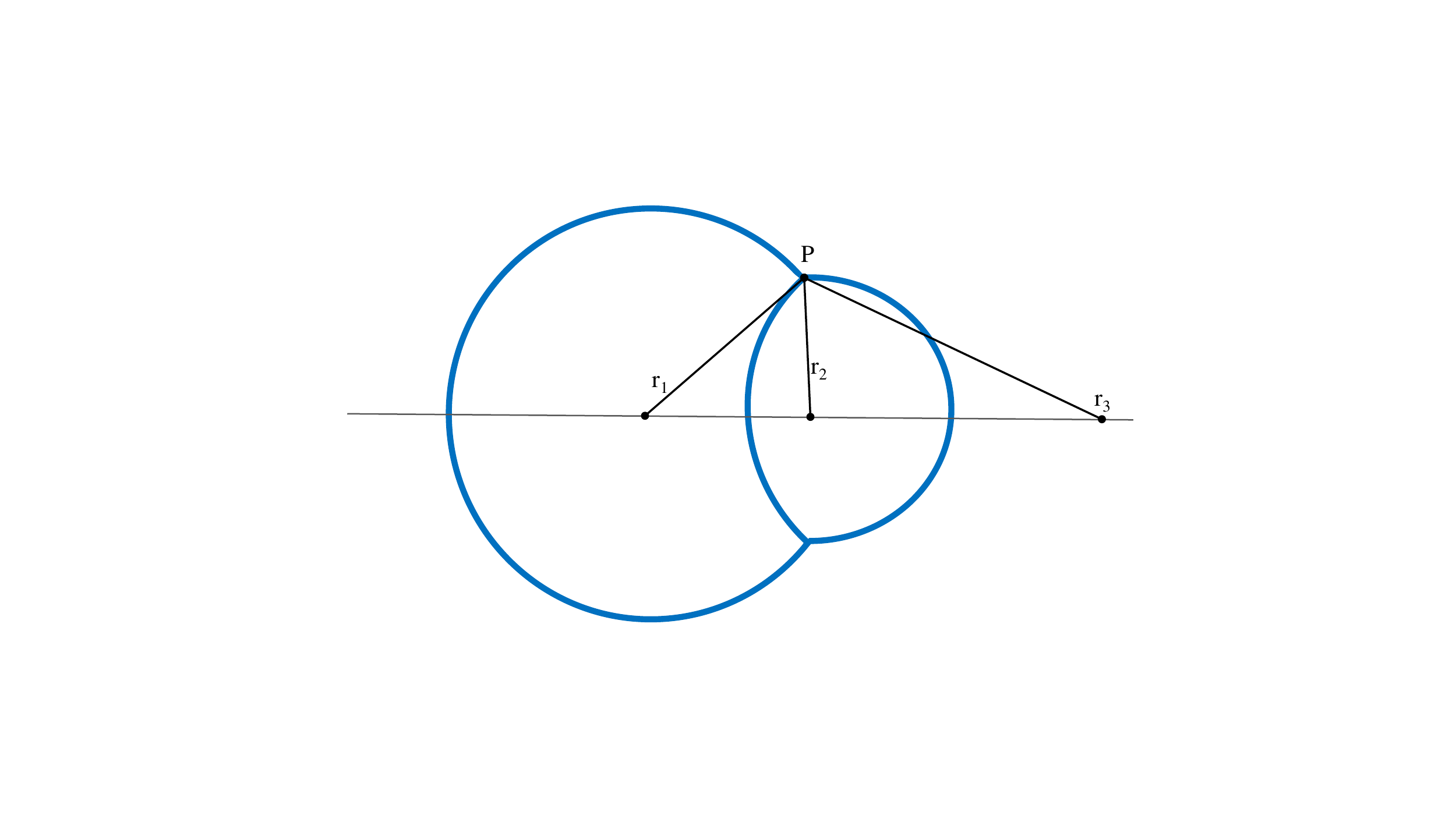}}
\begin{center}
\caption{\small{Double bubble foam.}}\label{doubleBoubble}
\end{center}
\end{figure}

\subsection{Lower bound for the area of a cell of a Kelvin foam}\label{secKelvin}
The accuracy of the lower bounds for the area of a domain of a foam  given by any method based on the main theorem, relies drastically on the number and the distribution of vertex (or face, or edge) points that are chosen. If we choose very few points the inequalities given by the main theorem are far away to be sharp, but we can nevertheless obtain a lower bound.  In this subsection, as an example of an application, we will show how to use the main theorem directly to obtain a lower bound for the area of a foam. In this case we will use a  cell of the Kelvin foam (where lower bounds for the area function are known).

The Kelvin foam arises from the problem of partitioning space into equal-volume cells, using the least interface area.  Kelvin described his foam as a relaxation of the Voronoi partition for the body-centered cubic lattice, whose cells are congruent truncated octahedra. As the foam relaxes, see \cite{Kusner96}, the symmetry present in the Voronoi partition is preserved and hence the vertices are fixed by symmetry. The square faces remain in their mirror planes, although the edges bend within these planes. The diagonals of the hexagons remain fixed along axes of rotational symmetry, while the hexagons become shaped like monkey-saddles (see \cite{Kusner96} and figure \ref{extball}).  Observe that the resulting foam is a minimal foam, i.e. it has $h=0$.

In this subsection we are interested in a lower bound for the contribution of each cell to the total area of the foam. Observe that each face contributes in half part of its area to the total area of the foam (since each face is in contact with two cells). We consider an extrinsic disc centered at the center of each hexagonal face (these centers remain fixed under the foam relaxation). If we denote by $a$ the (extrinsic) distance between two adjacent vertices, and if we choose the radii of such extrinsic discs as $R_1=\frac{a}{2}$ (less than $\frac{3}{4}a$, the maximal distance between the center of the hexagons and the planes where the square faces lie), and apply the main theorem, then we obtain the following lower bound for the area of these extrinsic foam discs:
\begin{equation}
A_1=\overset{\text{number of hexagonal faces}}{8}\cdot \frac{1}{2}\cdot \pi \left(\frac{a}{2}\right)^2
\end{equation}

Now we can put an extrinsic disc of radius $R_2=\frac{a}{2}$ in each vertex, but now these discs contribute only the fourth part to the total area because identical cells meet at fours at vertices, and this give us the following lower bound
\begin{equation}
A_2=\overset{\text{number of vertices}}{24}\cdot \frac{1}{4}\cdot 3{\rm arccos}(-1/3)\left(\frac{a}{2}\right)^2
\end{equation}

Since the square faces remain in their mirror planes but with edges bent within these planes, the original square (of area $a^2$) is contained in the final relaxed square face. Then, we obtain the lower bound
\begin{equation}
A_3=\overset{\text{number of square faces}}{6}\cdot \frac{1}{2}\left(a^2-\pi\left(\frac{a}{2}\right)^2\right)
\end{equation}
where we have subtracted the (yet computed) contribution of the discs of the four vertices. Finally we obtain
\begin{equation}
\text{ Area of a cell of the Kelvin foam}\geq A_1+A_2+A_3\approx 12.3832 a^2.
\end{equation}

But observe that by using a ``slicing argument'' the best lower bound estimation is obtained in \cite{Kusner96} as
\begin{equation}
\text{ Area of a cell of the Kelvin foam}>6\left(\sqrt{\frac{3}{2}}+1\right)a^2\approx 13.3485 a^2.
\end{equation}

\subsection{Lower bounds for the cost function}\label{secCost}
Given a foam $F\subset \erre^3$, and a domain $\Omega\subset \erre^3$ such that $F\cap \Omega\neq \emptyset$ (with perhaps $\Omega=\erre^3$), studying the scale-invariant \emph{cost} function we get -- from an isoperimetric point of view -- an estimate for how effective the foam fills the domain $\Omega$. Here
\begin{equation}
\mu(\Omega):=\frac{(\textrm{A}(F\cap \Omega)/n)^3}{(\textrm{V}(\Omega)/n)^2}
\end{equation}
where $n$ denotes the number of cells inside $\Omega$. In the particular case of a minimal foam we can state a lower bound for this cost function reading off only information about the vertices and their distribution.

Given such a domain $\Omega$, let us use the following notation:
\begin{equation}
\begin{aligned}
n&:=\text{ number of cells inside }\Omega\\
v_\Omega&:= \text{ number of vertices inside }\Omega\\
\overline v_\Omega&:= \frac{v_\Omega}{n}\\
d&:=\text{ minimal extrinsic distance between vertices}\\
\nu_{\Omega}&:=\frac{v_\Omega}{{\rm V}(\Omega)}
\end{aligned}
\end{equation}

\begin{corollary}Let $F\subset \erre^3$ be a minimal foam in $\erre^3$, hence for any domain  $\Omega\subset \erre^3$
\begin{equation}
\mu(\Omega)\geq \overline v_\Omega \nu_\Omega^2A_0^3(d),
\end{equation}
where $A_0$ is given by
$$
A_0(d)=\theta_v \pi \left(\frac{d}{2}\right)^2, \quad \theta_v=\frac{3}{\pi}{\rm arccos}\left(-\frac{1}{3}\right).
$$
\end{corollary}
\begin{proof}
By using the main theorem,
\begin{equation}
\begin{aligned}
\mu(\Omega)= \frac{1}{n}\frac{(\textrm{A}(\Omega))^3}{(\textrm{V}(\Omega))^2}\geq \frac{1}{n}\frac{(v_\Omega \theta_v \pi d^2)^3}{(\textrm{V}(\Omega))^2} =\left(\frac{v_\Omega}{n}\right)\left(\frac{v_\Omega }{\textrm{V}(\Omega)}\right)^2(\theta_v \pi d^2)^3.
\end{aligned}
\end{equation}

\end{proof}

A classical related problem arising from Kelvin's problem is concerned with the study of the optimization of the cost function under partitions of the space in periodic domains (see \cite{Kusner96}). A \emph{partition} of space is a division of $\erre^3$ into disjoint cells. We are mainly interested in the surfaces forming the interface between cells. The partitions that we consider will be periodic with respect to some lattice, with some number $n$ of cells in each periodic domain. Obviously, we can use periodic minimal foams as examples of such partitions of space.


In such a case by using the lower bound given in \cite{Kusner92} for the average number of vertices in a periodic minimal foam,
$$
\overline v_\Omega \geq 24,
$$
\begin{corollary}
Let $F\subset \erre^3$ be a minimal and periodic foam in $\erre^3$, hence for any domain  $\Omega\subset \erre^3$
\begin{equation}
\mu(\Omega)\geq 24\, \nu_\Omega^2A_0^3(d).
\end{equation}
\end{corollary}

\subsection{Lower bound for the pressure on minimal foams}\label{secPreasure}
A problem which is related to the isoperimetric properties is concerned with the pressure distribution in a foam. We consider a dried foam in mechanical equilibrium with an ideal gas in the bubble chambers. The gas in each chamber is assumed to satisfy the ideal gas equation of state and we assume that the entire foam is held at constant temperature. In such a case the ratio of area per volume is fixed by the physical constraints involving the temperature, pressure, and surface tension. Indeed, if we have a minimal foam $F\subset \Omega$ inside of a region $\Omega\subset \erre^3$, since the foam is minimal, each cell is at the same pressure $p_{\rm in}$ but there is also an extrinsic pressure $p_{\rm ext}$ coming from the boundary $\partial \Omega$ and those two pressures are related as follows (see \cite{Ross69,Aref2000, Fortes2001, Fortes2003, Fortes2007}):
\begin{equation}
p_{\rm int}-p_{\rm ext}=\frac{3}{2} \sigma \frac{\textrm{A}(F)}{\textrm{V}(\Omega)}
\end{equation}
where $\sigma$ is the surface tension assumed to be homogeneous throughout the foam. Hence,
\begin{equation}
p_{\rm int}-p_{\rm ext}=\frac{3}{2} \sigma \left(\frac{n}{\textrm{V}(\Omega)}\mu(\Omega)\right)^\frac{1}{3}
\end{equation}
The pressure inside of each cell in the foam can be therefore estimated from bellow using only the vertices as
\begin{equation}
p_{\rm int}\geq p_{\rm ext}+\frac{3}{2} \sigma \nu_\Omega A_0(d).
\end{equation}

\subsection{Area and Extrinsic Vertex Area}\label{secEva}
The main result allows us to construct an algorithm to estimate lower bounds for the area of any  domain $\Omega\subset F$ of a foam $F$. Given a finite set of points $X = \{x_{1}, x_{2}, \cdots \,x_N\} \subset \Omega$. This set $X$ has associated the following subset of $\erre^N$
\begin{equation}\label{Omega}
\begin{aligned}
D_\Omega(X):=\Big\{& (r_1,r_2,\cdots,r_N)\subset \erre^N\,\vert\,B_{r_i}(x_i)\subset \Omega\,\textrm{ for }i=1,\cdots,N\\
& \textrm{ and} \quad B_{r_{i}}(x_{i}) \cap B_{r_{j}}(x_{j}) = \emptyset  \,\text{ for all }\, i \neq j\Big\}.
\end{aligned}
\end{equation}
We define the extrinsic vertex area of $X$ in $\Omega$ as follows:


\begin{definition}\label{def1} Given a set of isolated points $X$ in $\Omega\subset\erre^3$, the \emph{lower extrinsic vertex area} of $X$ with $h$ weight  is given by
\begin{equation}
\begin{aligned}
\eva(X,\Omega,h) := \max_{(r_1,\cdots,r_N)\in D_\Omega(X)}  \left\{ \sum_{i} e^{-2hr_i}\theta_{v}\cdot\pi r_{i}^{2} \right\}
\end{aligned}
\end{equation}
where $D_\Omega(X)$ is given by definition (\ref{Omega}) and $\theta_v=\frac{3}{\pi}\arccos(-1/3)$ . The \emph{extrinsic vertex area}  of $X$ with $h$ weight is given by
\begin{equation}
\evA(X, \Omega,h):= \max_{Z \, \subset \, X}\left\{ \eva(Z, \Omega,h) \right\}.
\end{equation}
\end{definition}

Observe that the extrinsic vertex area associates an area to a set of points. By using our  main theorem we can make use of the extrinsic vertex area associated to a set of points $X\subset\Omega$ to obtain a lower bound for the area of a piece of the foam $\Omega\subset F$.

\begin{theorem}\label{teo5.2}
Let $F\subset \erre^3$ be a foam and let $\Omega\subset F$ be a compact subset of $F$. Denote by $\mathcal{V}_\Omega$ the set of vertices of $\Omega$, and denote by
$$
h=\max_{x\in \Omega}\Vert\vec H\Vert(x)
$$
the maximum  of the norm of the mean curvature vector field on the faces of $\Omega$. Then, the area of $\Omega$ is bounded from below by
\begin{equation}
\textrm{A}(\Omega)\geq \evA(X,\mathcal{V}_\Omega,h),
\end{equation}
where $\evA(X,\mathcal{V}_\Omega,h)$ is the extrinsic vertex area of the vertices points $\mathcal{V}_\Omega$ with $h$ weight.
\end{theorem}

\begin{remark}
Observe moreover, that since by definition
\begin{equation}
\evA(X,\Omega,h) \geq \eva(X, \Omega,h)
\end{equation}
then, under the hypothesis of the above theorem we can state,
\begin{equation}
\textrm{A}(\Omega)\geq \eva(X,\mathcal{V}_\Omega,h).
\end{equation}
\end{remark}
\begin{remark}
Given a set of $N$-isolated points $X$ into  $\Omega\subset\erre^3$, the corresponding $D_\Omega(X)$ can be understood as the set of admissible radii in order to place the extrinsic foam discs to estimate the largest lower bound for the area of such set of extrinsic discs. But observe that since the function
$$
f:\erre\to\erre,\quad r\to f(r)=e^{-2hr}r^2
$$
is increasing only in the interval $[0,\frac{1}{h})$ (with $[0,\infty)$ if $h=0$) then we can define the \emph{maximal radius} $r_{\rm max}:=\frac{1}{h}$ and the \emph{maximal hypercube} $B^\infty_{r_{\rm max} }:=[0,r_{\rm max}]^N$ in such a way that
$$
\eva(X,\Omega,h)=\max_{(r_1,\cdots,r_N)\in D_\Omega(X)\cap B^\infty_{r_{\rm max} }}  \left\{ \sum_{i} e^{-2hr_i}\theta_{v}\cdot\pi r_{i}^{2} \right\}
$$
Stated otherwise, the maximum in definition \ref{def1} is obtained in
$$
(r_1,\cdots,r_N)\in D_\Omega(X)\quad\text{with }\, r_i\leq r_{\rm max}=\frac{1}{h}\quad \text { for } i=1,\cdots,N.
$$
The presence of mean curvature on the faces of a foam forces therefore our method to make use only of small extrinsic discs.
\end{remark}

\subsubsection{From the matrix of distances to the extrinsic vertex area}
Let now $X=\{x_1,\cdots,x_N\}$ be a set of $N$ isolated points inside a compact subset $\Omega\subset \erre^3$ with boundary $\partial \Omega$.  We can define the \emph{matrix of distances} $\mathcal{D}_\Omega(X)$ of $X$ in $\Omega$ to be the $(N+1)\times( N+1)$ matrix  given by
$$
\left[\mathcal{D}_\Omega(X)\right]_{i,j}:=
\left\{\begin{array}{cclcr}
\dist_{\erre^3}\left(x_i,x_j\right)&{\rm if}&i\leq N&{\rm and}&j\leq N,\\
\dist_{\erre^3}\left(x_i,\partial \Omega\right)&{\rm if}&i\leq N&{\rm and}&j= N+1,\\
\dist_{\erre^3}\left(x_j,\partial \Omega\right)&{\rm if}&i=N+1&{\rm and}&j\leq N,\\
0&{\rm if}& i=N+1&{\rm and}&j=N+1.
\end{array}
\right.
$$

From this  matrix we can define as well the \emph{reduced  matrix of distances} to be the $N\times N$ matrix given by
$$
\left[d_\Omega(X)\right]_{i,j}:=\min\left\{ \left[\mathcal{D}_\Omega(X)\right]_{i,j},\left[\mathcal{D}_\Omega(X)\right]_{i,N+1},\left[\mathcal{D}_\Omega(X)\right]_{j,N+1},\frac{1}{h} \right\}
$$
if $N>1$, or by
$$
\left[d_\Omega(X)\right]_{1,1}:=\min\left\{ \dist_{\erre^3}(x_1,\partial \Omega),\frac{1}{h} \right\}
$$
if $N=1$. Observe that
\begin{equation}\label{polytope}\begin{array}{llr}
D_\Omega(X)\cap B^\infty_{r_{\rm max} }=\bigg\{r\in \erre^N\,\vert\,&  0\leq r_i+r_j\leq \left[d_\Omega(X)\right]_{i,j}, &\forall i \text{ and } \forall j\neq i\bigg\}
\end{array}
\end{equation}
if $N>1$, and
$$
D_\Omega(X)\cap B^\infty_{r_{\rm max} }=\left[0, \left[d_\Omega(X)\right]_{1,1}\right]
$$
if $N=1$.
By using inequalities (\ref{polytope}), $D_\Omega(X)\cap B^\infty_{r_{\rm max} }$ is a polytope. But,  since
$$
\eva:\erre^N\to\erre,\quad \eva(\vec z)=\sum_{i=1}^N\pi e^{-2hz_i}\theta_vz_i^2
$$
has no critical points in $D_\Omega(X)\cap B^\infty_{r_{\rm max} }\setminus\{\vec 0\}$, the maximum of the definition of the extrinsic vertex area is attained in the boundary
$$
\eva(X,\Omega,h)=\max_{\vec r\in \partial \left(D_\Omega(X)\cap B^\infty_{r_{\rm max} }\right)}\left\{\sum_{i=1}^N\pi e^{-2hr_i}\theta_vr_i^2\right\}.
$$
and hence,
$$
\evA(X,\Omega,h)=\max_{Z\subset X }\left(\max_{\vec r\in \partial \left(D_\Omega(Z)\cap B^\infty_{r_{\rm max} }\right)}\left\{\sum_{i=1}^N\pi e^{-2hr_i}\theta_vr_i^2\right\}\right),
$$
where $Z$ ranges on subsets of $X$. Moreover if
$$
\left[d_\Omega(X)\right]_{i,j}\leq \frac{2-\sqrt{2}}{h}, \quad \forall i,\forall j\quad,
$$
the function $\eva$ is a convex function in the polytope $D_\Omega(X)\cap B^\infty_{r_{\rm max} }$, and hence its maximum is attained in the vertices of $D_\Omega(X)\cap B^\infty_{r_{\rm max} }$. Then, in the case $\left[d_\Omega(X)\right]_{i,j}\leq \frac{2-\sqrt{2}}{h}$ (minimal foams for instance) an estimation of the $\evA$ can be obtained using the \emph{vertex enumeration problem}. In section \ref{num-alg} we will provide an algorithm using precisely the approach of the vertex enumeration problem.

 Observe that since $Z$ is a subset of $X$, then the matrix of distances $\mathcal{D}_\Omega(Z)$ can be obtained from the matrix $\mathcal{D}_\Omega(X)$ removing several rows and the corresponding columns.

\begin{figure}
\centerline{
\includegraphics[scale=0.27]{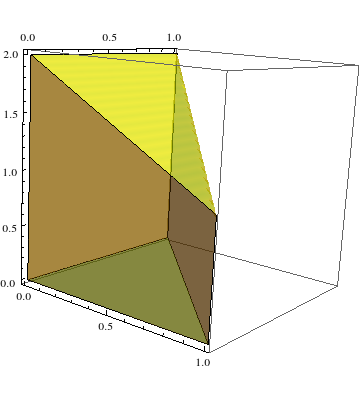}}
\begin{center}
\caption{\small{Domain $D_\Omega(X)$ from example \ref{example-matrix}.}}\label{example}
\end{center}
\end{figure}

\begin{example}\label{example-matrix}
Let $\Omega$ be a domain  inside a minimal foam ($h=0$) such that the  matrix of distances of the vertices $X=\{A,B,C\}$ of $\Omega$ is
$$
\mathcal{D}_\Omega(X)=\begin{pmatrix}
0&1&2&4\\
1&0&3&4\\
2&3&0&4\\
4&4&4&0
\end{pmatrix}
$$
The following graph can be useful in order to represent the distances between vertices and the distances between vertices and the boundary $\partial \Omega$,
\begin{center}
\resizebox{80pt}{80pt}{
\begin{tikzpicture}[scale=0.2,>=stealth',shorten >=1pt,node distance=3cm,on grid,initial/.style    ={}]
  \node[state]          (B)     {$B$};
  \node[state]          (M) [below left =of B]    {$A$};
  \node[state]          (D) [above  =of B]    {$\partial \Omega$};
  \node[state]          (C) [below right =of B]    {$C$};
\tikzset{mystyle/.style={-,double=orange}}
\tikzset{every node/.style={fill=white}}
\path (C)     edge [mystyle]    node   {$3$} (B)
      (B)     edge [mystyle]   node   {$1$} (M)
      (C)     edge [mystyle]   node   {$2$} (M);
\tikzset{mystyle/.style={-,double=green}}
\path (C)     edge [mystyle]   node   {$4$} (D)
      (D)     edge [mystyle]    node   {$4$} (B)
      (D)     edge [mystyle]   node   {$4$} (M);
\end{tikzpicture}
}
\end{center}

\noindent Then, the reduced matrix of $X=\{A,B,C\}$ is just
$$
d_\Omega(X)=\begin{pmatrix}
0&1&2\\
1&0&3\\
2&3&0
\end{pmatrix}
$$
and the domain $D_\Omega(X)\cap B^\infty_{r_{\rm max}}$ is the domain bounded (see figure \ref{example}) by the  following inequalities $M . r \leq b$
\begin{equation}\label{inequality}
\left\{\begin{array}{ccc}
r_1+r_2&\leq& 1,\\
r_1+r_3&\leq& 2,\\
r_2+r_3&\leq& 3,\\
-r_1&\leq&0,\\
-r_2&\leq&0,\\
-r_3&\leq&0.
\end{array}
\right.
\end{equation}
where $M$ is a matrix of zeros and ones relating the different radii among themselves and with the boundary  $\partial \Omega$ and $b$ is the vector of distances.
The vertices of the polytope $D_\Omega(X)\cap B^\infty_{r_{\rm max}}$ are
$$
\left\{(0,0,0),(1,0,0),(0,1,0),(0,1,2),(0,1,1),(1,0,1)    \right\}.
$$
It is not hard to see that
$$
\eva(X,\Omega,0)=(0^2+1^2+2^2)\theta_v\pi=5\theta_v\pi.
$$
If we choose $Z_1=\{B,C\}$, then
$$
\mathcal{D}_\Omega(Z_1)=\begin{pmatrix}
0&3&4\\
3&0&4\\
4&4&0\\
\end{pmatrix}
$$
with
$$
d_\Omega(Z_1)=\begin{pmatrix}
0&3\\
3&0
\end{pmatrix}
$$
and $\eva(Z_1,\Omega,0)=9\theta_v\pi$. But we can check that
$$
\evA(X,\Omega,0)=\eva(\{A\},\Omega,0)=\eva(\{B\},\Omega,0)=\eva(\{C\},\Omega,0)=16\theta_v\pi.
$$
\end{example}

\subsubsection{Numerical solution of \textbf{evA}}\label{num-alg}
Consider for example, that we have a domain $\partial \Omega$ with shape of sphere and radius $R =10$. In this domain we introduce eight points with random distribution forming the vertices of the foam. In Figure \ref{evA} (a), we show the maximum value of $eva$ found $\eva(X,\Omega,0)=29.0257 \theta_v\pi$  from the set of inequalities $M.r \leq b$ as it appears in the expression \ref{inequality}. To solve the problem, the Interior-Point algorithm has been used with the Optimization Tool in MATLAB.

On the other hand, we can define a domain $D_\Omega(X)\cap B^\infty_{r_{\rm max}}$ in the radii dual space corresponding to all the solutions that can be found from the set of inequalities. The set of vertices $V_i$ of this domain form a polytope in which one of the vertices maximizes $eva$. Applying the algorithm proposed in Algorithm \ref {algorithm 1}, it is possible to estimate a solution for $evA$ with the minimum of vertices needed (see Figure \ref{evA} (b)). The algorithm is based in the algorithm proposed in \cite{Kleder} that convert convex constraint inequalities into a set of vertices. In this example, only three of the initial vertices survive in the final solution  $\evA(X,\Omega,0)=36.4422 \theta_v\pi$, while the rest of vertices are absorbed.

\begin{algorithm}[H]
\caption{\emph{evA} calculation}\label{algorithm 1}
\begin{algorithmic}[1]
\BState \textbf{Input:}
\State $X$: set of $N$-isolated points into $\Omega\subset\erre^3$ with boundary $\partial \Omega$.
\BState \textbf{Output:}
\State $evA = 0$.
\BState \emph{loop}:
\State Estimate the distance matrix $D_\Omega(X)$ in $X$ with boundary $\partial \Omega$.
\State Build the set of inequalities $M.r \leq b$ (see \ref{inequality}).
\State Make an initial estimation of the radii $r=(r_1,\ldots,r_N)$ that maximize $eva$.
\State Re-estimated the distance vector $b = b - M.r$.
\State Calculate the normalized matrix $D_{ij} = M_{ij}/b_i$ by the distance vector $b$.
\State Converts the polytope (convex hull) defined by the system of inequalities $D*r\leq b$, in a list a vertices $V$ in the radii dual space, where the domain $D_\Omega(X)\cap B^\infty_{r_{\rm max}}$ is the bounded domain.
\State Find the vertex $V_i=(\acute{r}_1,\ldots,\acute{r}_N)$ in the polytope that maximize $eva$.
\If {$evA > eva$}
\State \Return \emph{break loop}.
\EndIf
\State Assign the value $evA = eva$.
\State Find the number $L$ of collapsed radii, such that $\acute{r}_i = 0$.
\If {$L > 0$}
\State Remove the $L$ vertices with radius $r_i = 0$ and $N \gets N-L$.
\BState \textbf{else}
\State Remove the vertex with minimum radius and $N \gets N-1$.
\EndIf
\BState \emph{end loop}
\end{algorithmic}
\end{algorithm}

\begin{figure}
\centerline{
\includegraphics[scale=0.25]{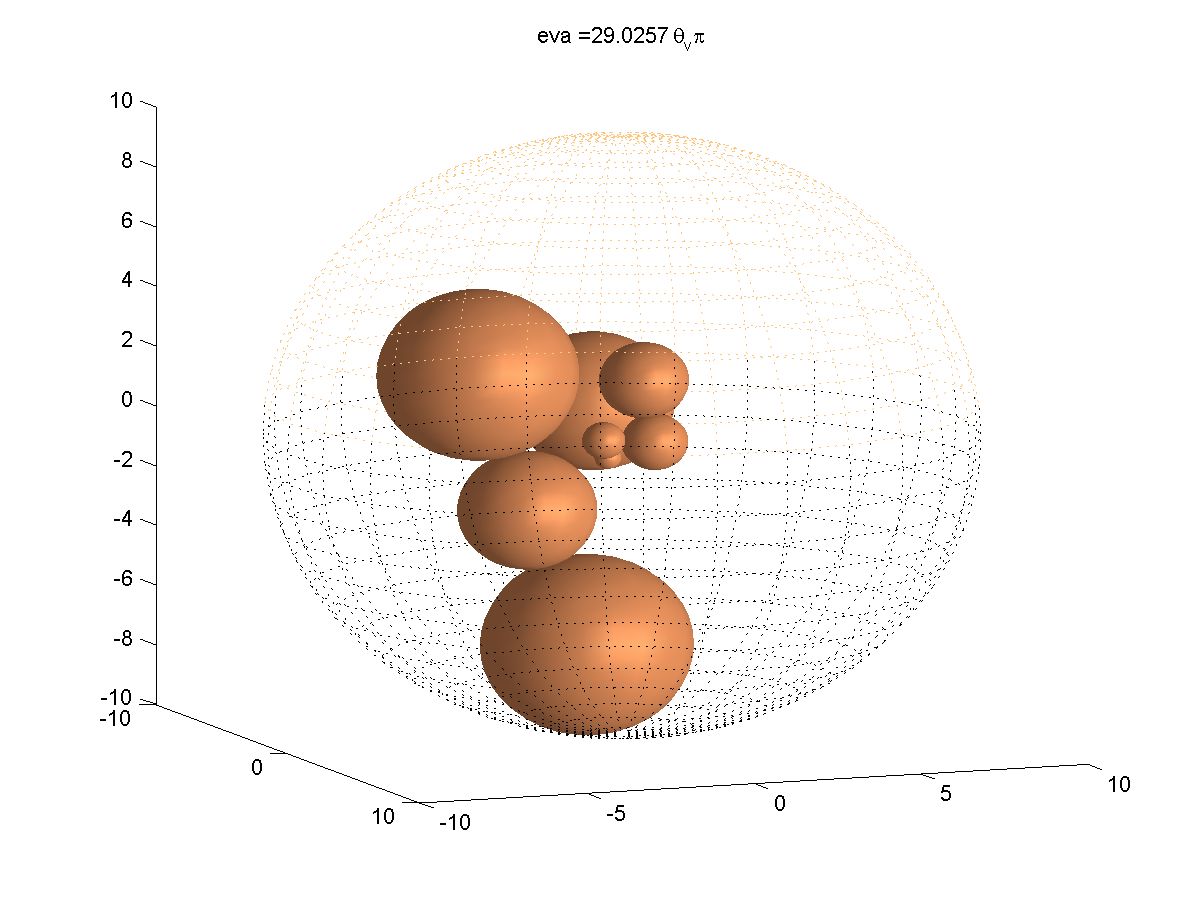} \includegraphics[scale=0.25]{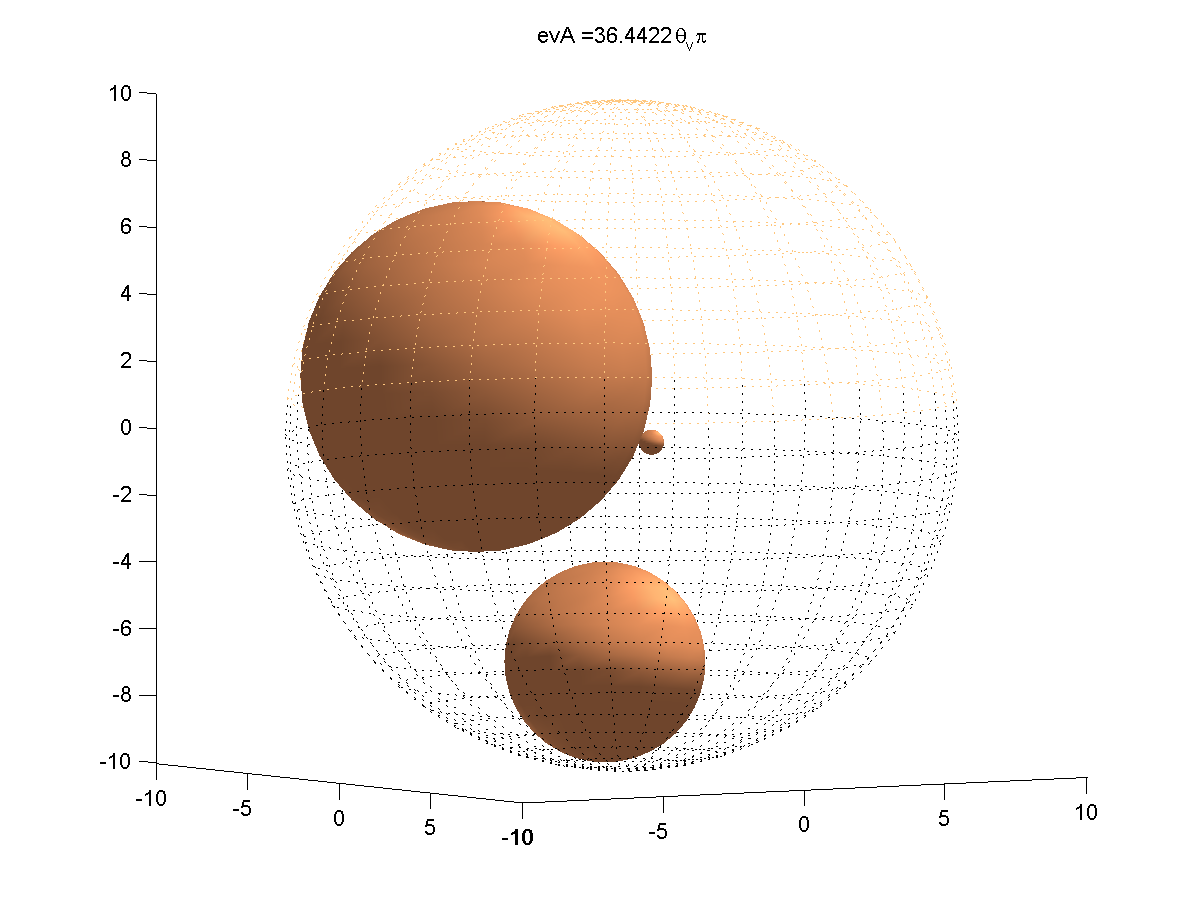}
}
\begin{center}
\caption{\small{(a) We show the solution that maximizes $eva$ for the case of eight vertices inside the sphere domain $\partial \Omega$. (b) We show the distribution of spheres that maximizes $evA$}}\label{evA}
\end{center}
\end{figure}











\def\cprime{$'$} \def\cprime{$'$} \def\cprime{$'$} \def\cprime{$'$}
  \def\cprime{$'$}
\providecommand{\bysame}{\leavevmode\hbox to3em{\hrulefill}\thinspace}
\providecommand{\MR}{\relax\ifhmode\unskip\space\fi MR }
\providecommand{\MRhref}[2]{%
  \href{http://www.ams.org/mathscinet-getitem?mr=#1}{#2}
}
\providecommand{\href}[2]{#2}

\bibliographystyle{siam}

\end{document}